\documentclass[11pt]{amsart}

\usepackage{latexsym,amsfonts,amssymb,exscale,enumerate,comment}
\usepackage{amsmath,amsthm,amscd}
\usepackage[all,knot]{xy}
\usepackage{fullpage, braket}
\usepackage{caption}
\usepackage{subcaption}

\newcommand{\wt}[1]{{\widetilde{#1}}}

\usepackage{url}
\usepackage[bookmarks=true,%
    colorlinks=true,%
    linkcolor=blue,%
    citecolor=blue,%
    filecolor=blue,%
    menucolor=blue,%
    urlcolor=blue,%
    breaklinks=true]{hyperref}

\input xy
\usepackage[all]{xy}
\SelectTips{cm}{}

\usepackage{tikz}
\usepackage{tikz-cd} 
\usetikzlibrary{shapes,snakes}
\usetikzlibrary{decorations.markings}
\usetikzlibrary{decorations.pathreplacing}
\tikzstyle directed=[postaction={decorate,decoration={markings,
    mark=at position #1 with {\arrow{>}}}}]

\newcommand{\hackcenter}[1]{
 \xy (0,0)*{#1}; \endxy}

\tikzset{->-/.style={decoration={
  markings,
  mark=at position #1 with {\arrow{>}}},postaction={decorate}}}

\usepackage{graphicx}
\usepackage{color}
\newcommand{\brk}[1]{{\left\langle{#1}\right\rangle}}
\newcommand{\bp}[1]{{\left({#1}\right)}}
\newcommand{\bs}[1]{{\left\{{#1}\right\}}}
\renewcommand{\dag}[1]{{#1}^\dagger}
\newcommand{\ldag}[1]{{^\dagger#1}}
\newcommand{\dagp}[1]{\bp{#1}^\dagger}

\renewcommand{\mod}{\operatorname{mod}}

\newcommand{\SVect}{\operatorname{SVect}}
\newcommand{\tcoev}{\stackrel{\longleftarrow}{\operatorname{coev}}}
\newcommand{\tev}{\stackrel{\longleftarrow}{\operatorname{ev}}}
\newcommand{\ev}{\stackrel{\longrightarrow}{\operatorname{ev}}}
\newcommand{\coev}{\stackrel{\longrightarrow}{\operatorname{coev}}}
\newcommand{\sqr}{\operatorname{sqr}}

\newcommand{\Proj}{{\mathsf {Proj}}}
\newcommand{\vp}{\varphi}

\newcommand{\dt}{\sqrt\theta}
\newcommand{\ellr}{r}

\newcommand{\Graph}{\operatorname{Gr}}

\newcommand{\qn}[1]{{\left\{#1\right\}}}
\newcommand{\ob}{\operatorname{Ob}(\cat)}

\newcommand{\y}{\operatorname{\mathsf{y}}}
\newcommand{\md}{\operatorname{\mathsf{d}}}
\newcommand{\mb}{\operatorname{\mathsf{b}}}
\newcommand{\mt}{\operatorname{\mathsf{t}}}
\newcommand{\rot}{\operatorname{\mathrm{rot}}}

\newcommand{\et}{{\quad\text{and}\quad}}

\newcommand{\id}{{\rm id}}
\newcommand{\Gr}{\mathcal{G}}
\newcommand{\X}{\mathcal{X}}
\newcommand{\XX}{{\mathsf X}}
\newcommand{\col}{{\Phi}}
\newcommand{\HHH}{\ensuremath{\mathcal{H}}}
\newcommand{\St}{\operatorname{St}}
\newcommand{\cat}{\mathcal{C}}

\newcommand{\catdHem}{\mathcal{D}^{Herm}}
\newcommand{\catdSeq}{\mathcal{D}^{Sesq}}

\newcommand{\FK}{{\Bbbk}}
\newcommand{\I}{\mathcal{I}}
\newcommand{\m}{\mathsf{m}}
\newcommand{\mirror}[1]{{\stackrel{\leftrightarrow}{#1}}}

\newcommand{\unit}{\ensuremath{\mathbb{I}}}
\newcommand{\con}[1]{\bar{#1}^{\dagger}}
\newcommand{\Zt}{{\mathsf Z}}
\newcommand{\roots}{\Delta}
\newcommand{\p}[1]{\ensuremath{\bar {#1}}}

\theoremstyle{plain}
\newtheorem{theorem}{Theorem}
\newtheorem{theorem_noname}{Theorem}
\newtheorem{corollary}[theorem]{Corollary}
\newtheorem{proposition}[theorem]{Proposition}
\newtheorem{lemma}[theorem]{Lemma}

\theoremstyle{definition}

\newtheorem{definition}[theorem]{Definition}
\newtheorem{conjecture}[theorem]{Conjecture}
\theoremstyle{definition}
\newtheorem{remark}[theorem]{Remark}

\numberwithin{equation}{section}
\numberwithin{theorem}{section}

\newcommand{\maps}{\colon}

\newcommand{\refequal}[1]{\xy {\ar@{=}^{#1}
(-1,0)*{};(1,0)*{}};
\endxy}

\newcommand{\Hom}{{\rm Hom}}

\renewcommand{\to}{\rightarrow}

\def\rk{{\mathrm{rk}}}
\def\Id{\mathrm{Id}}

\def\mf{\mathfrak}

\def\T{{\mathcal{T}}}

\numberwithin{equation}{section}

\newcommand{\wb}{\overline}
\newcommand{\de}[1]{|#1|}
\newcommand{\sig}[2]{(-1)^{|#1||#2|}}

\newcommand{\qd}{\operatorname{\mathsf{d}}}

\newcommand{\HR}{\ensuremath{{\mathcal H}}}

\let\tilde=\widetilde

\let\epsilon=\varepsilon

\usepackage{bbm}
\def\C{{\mathbb{C}}}
\def\N{{\mathbbm N}}
\def\R{{\mathbbm R}}
\def\Z{{\mathbbm Z}}
\def\H{{\mathcal{H}}}

\def\cal#1{\mathcal{#1}}%
\def\1{\mathbbm{1}}%
\def\tr{\mathrm{tr}}%
\def\ptr{\mathrm{ptr}}%

\def\la{\langle}
\def\ra{\rangle}

\newcommand{\Gfun}{\ensuremath{\mathsf{G}}}

\newcommand{\ang}[1]{{\left\langle{#1}\right\rangle}}

\newcommand{\A}{{\sf A}}

\newcommand{\bb}{{\sf b}}
\newcommand{\G}{\cal{G}}

\newcommand{\PGr}{Z}
\newcommand{\Su}{\Sigma}
\newcommand{\dSu}{\widetilde{\Sigma}}
\newcommand{\coh}{\omega}
\newcommand{\La}{{\mathcal{L}}}
\newcommand{\dM}{\wt M}
\newcommand{\Zr}{{\mathsf N}}

\newcommand{\lk}{\operatorname{lk}}
\newcommand{\hS}{\widehat{S}}
\newcommand{\V}{\mathsf{V}}
\newcommand{\Span}{\operatorname{Span}}
\newcommand{\VV}{\mathbb{V}}

\newcommand{\calD}{\mathcal{D}}

\usepackage{bbm}

\def\cal#1{\mathcal{#1}}

\newcommand\nc{\newcommand}
\nc\rnc{\renewcommand}
\nc\Kar{\operatorname{Kar}}
\nc\End{\operatorname{End}}

\newcommand{\scs}{\scriptstyle}

\nc\Sym{\operatorname{Sym}}

\allowdisplaybreaks

\title {Non-semisimple Levin-Wen Models and Hermitian TQFTs
from quantum (super)groups}

\begin{document}

\author[N. Geer]{Nathan Geer}
\address{Mathematics \& Statistics\\
  Utah State University \\
  Logan, Utah 84322, USA}
  \email{nathan.geer@gmail.com}
\author[A.D. Lauda]{Aaron D. Lauda}
\address{Mathematics \& Physics \\
 University of Southern California \\
  Los Angeles, California 90089, USA}
  \email{lauda@usc.edu}
\author[B. Patureau-Mirand]{Bertrand Patureau-Mirand}
\address{UMR 6205, LMBA, universit\'e de Bretagne-Sud,
  BP 573, 56017 Vannes, France }
\email{bertrand.patureau@univ-ubs.fr}
\author[J. Sussan]{Joshua Sussan}
\address{Department of Mathematics\\
  CUNY Medgar Evers \\
  Brooklyn, NY 11225, USA}
  \email{jsussan@mec.cuny.edu}
  \address{Mathematics Program\\
 The Graduate Center, CUNY \\
  New York, NY 10016, USA}
  \email{jsussan@gc.cuny.edu}
\begin{abstract}
We develop the categorical context for defining Hermitian non-semisimple TQFTs.   We prove that relative Hermitian modular categories give rise to modified Hermitian WRT-TQFTs and provide numerous examples of these structures coming from the representation theory of quantum groups and quantum superalgebras.  The Hermitian theory developed here for the modified Turaev-Viro TQFT is applied to define new pseudo-Hermitian topological phases that can be considered as non-semisimple analogs of Levin-Wen models.
  \end{abstract}

  \maketitle

\setcounter{tocdepth}{3}
\tableofcontents

\section{Introduction}
There is a rich interplay between two-dimensional topological phases in quantum mechanical systems and topological quantum field theories (TQFTs).  This interaction is further enriched as topological structures inherent in TQFTs lead to novel features, such as non-abelian braiding statistics for low energy excitations when expressed in the corresponding quantum mechanical models.   This exciting interplay has led to a number of proposals for realizing fault-tolerant quantum computation by exploiting the topological features arising in these models~\cite{FKLW,GTKLT,KMR2608953,Kit,LW}, what is collectively called topological quantum computation (TQC).

Both the Turaev-Viro (TV) TQFTs \cite{TV} and the Witten-Reshetikhin-Turaev (WRT) TQFTs \cite{RT, Witten} have appeared in approaches to TQC and the study of topological phases.   All of these connections rely on unitarity of the TQFT to produce physically meaningful quantum mechanical models.
Both TV and WRT TQFTs have the common feature  that the categorical data defining them, spherical categories and modular tensor categories, respectively, are based on {\em semisimple categories}.  The purpose of this article is to build a theory of Hermitian structures on a new class of TQFTs based on {\em non-semisimple categories}; this key structure leads to new physically relevant  applications of non-semisimple TQFTs.

\subsection{Beyond semisimple unitary TQFTs}
The primary examples of semisimple categories used to define (2+1)-dimensional TQFTs come from the representation theory of quantum groups where the quantum parameter is set to a root of unity.  These categories are not semisimple, instead, a semisimplification procedure is performed that throws away an infinite number of representations, leaving only a finite collection remaining.  This process is essential as it removes representations  with vanishing quantum dimensions, which cannot be input into either the TV or WRT TQFTs.  Kirillov \cite{Kir96} and Wenzl \cite{Wen98} showed that the semisimple categories coming from the representation theory of quantum groups give rise to Hermitian and unitary categories in the sense of Turaev \cite{Tu}.

It was an open question how to construct a TQFT which retains non-semisimple information.  A crucial step in this direction was achieved in \cite{GPT1} (also see \cite{GKP1, GKP3}) where a modified trace was introduced, leading to non-trivial invariants of links colored by representations that classically led to trivial invariants.
The modified trace was utilized in \cite{GPT2} where a non-semisimple TV invariant was constructed using the notion of a relative spherical category.
A non-semisimple WRT invariant \cite{CGP1} and TQFT \cite{BCGP2} were constructed using the notion of a relative modular category.  An extended TQFT was formulated by De Renzi \cite{D17}.
These non-semisimple constructions use a group $\Gr$ as input.

%
%
%
%
%
The key examples of such non-semisimple categories have an infinite
number of non-isomorphic simple objects 
with vanishing quantum
dimensions.
Nevertheless, these
non-semisimple TQFTs have remarkable properties, often proving more
powerful than their semisimple analogs. For example, non-semisimple
TQFTs lead to mapping class group representations with the notable
property that the action of a Dehn twist has infinite order, and thus
the representations could be faithful, (see \cite{BCGP2}). This is in
contrast with the usual quantum mapping class group representations
where all Dehn twists have finite order and the representations are
thus not faithful.  Also, after projectivization, these TQFTs
correspond to the Lyubashenko projective mapping class group
representations given in \cite{Lyubashenko:1994tm}, (also see
\cite{derenzi2021mapping, DGP2}).

Given that the non-semisimple TQFTs appear to be much more powerful than their semisimple counterparts, it suggests that these TQFTs can have similarly impactful applications to physical phenomena and the construction of topological phases.  We mentioned above the importance of unitarity for applying TQFTs to physical phenomena in the usual approach, but mentioned our work here will be focused on Hermitian structures for non-semisimple categorical data.  The reason is that the non-semisimple TQFTs {\em do not} lead to Hermitian Hamiltonians.   Instead, what naturally arises are pseudo-Hermitian Hamiltonians.  Such quantum mechanical systems have received a great deal of interest in the physics literature, as it has been observed that many physically relevant systems exhibit non-Hermitian Hamiltonians~\cite{BROWER1978213,PhysRev.115.1390,ZAMOLODCHIKOV1991619} .  The pseudo-Hermicity condition ensures that these systems have positive spectrum, normalizable wave functions, and time evolution given by the exponential of the Hamiltonian that is self adjoint with respect to an indefinite inner product~\cite{Mostafazadeh_2002, Mostafazadeh_2002a,Most-unitary}.  These inner products are also invariant under the time-evolution generated by the Schrodinger equation.  Essentially, these theories are unitary, but for an indefinite inner product.

\subsection{Hermitian non-semisimple TQFTs}
The goal of this article is to develop a theory of Hermitian structures that can be applied to relative spherical categories and relative modular tensor categories, giving rise to Hermitian non-semisimple WRT invariants
 and allowing for the construction of pseudo-Hermitian topological phases from the non-semisimple Turaev-Viro theory (discussed more in the next section).  In particular, in the WRT setting we prove the following.

\begin{theorem_noname} (Theorem \ref{thm:hermWRT})
A relative Hermitian $\Gr$-modular category gives rise to a modified Hermitian WRT TQFT.
\end{theorem_noname}

Using ideas of \cite{Kir96, Tu, Wen98}, the authors of the current paper gave a proof of concept for this approach by placing a Hermitian structure on a non-semisimple category of representations coming from the unrolled quantum group for $\mathfrak{sl}_2$ \cite{GLPS1}.  This led to the first example of a Hermitian modified WRT TQFT.  Here we vastly generalize this result.  While many of the techniques developed for $\mf{sl}_2$ generalize, in  \cite{GLPS1}  we relied on an explicit description of the category for $\mathfrak{sl}_2$ which allowed us to use certain representations in the earlier works that are inaccessible here.
Thus, we must develop additional theory needed to deduce the relevant Hermitian data in a general Hopf algebraic setup giving rise to many new examples of Hermitian $\Gr$-modular categories.

\begin{theorem_noname} (Theorems \ref{T:DH-Hermitian-modular}, \ref{T:DHZ-Hermitian-spherical}, \ref{T:superex})
Finite-dimensional simple complex Lie algebras as well as $\mathfrak{sl}(m|n)$ provide examples of
relative Hermitian $\Gr$-modular categories and relative Hermitian $\Gr$-spherical categories.
\end{theorem_noname}



\subsection{Turaev-Viro TQFT and topological phases}

Turaev-Viro TQFTs appear prominently in Kitaev's toric code~\cite{Kit} and Levin-Wen's  construction \cite{LW} of two-dimensional topological phases using `string net' models defined on dual graphs of triangulated surfaces.  These models produce exactly solvable gapped Hamiltonians exhibiting a topologically degenerate ground state.  We refer to these physical systems as LW-models.  They are constructed from algebraic data we call LW-data, essentially defining a spherical category.     The topologically dependent ground state of the LW Hamiltonian gives rise to a topological invariant of the surface it is defined on, and is known to be isomorphic to the vector space assigned to the surface coming from the Turaev-Viro TQFT \cite{Kir-stringnet, KKR}.

In \cite{GLPS2}, the authors extended the results in \cite{GLPS1} by defining a Hermitian structure  on the  non-semisimple TV TQFT associated to the category of representations for the semi-restricted quantum group for $\mathfrak{sl}_2$, showing this category was a Hermitian $\Gr$-relative spherical category.  We also showed that using certain algebraic data, that we call $\Gr$-LW data, we could define a pseudo-Hermitian Levin-Wen model.  A number of simplifications made this example easier to define, including multiplicity one results in the relevant non-semisimple category.  Here we generalize this work by extending the definition of $\Gr$-LW data and proving the following.

\begin{theorem_noname} (Sections \ref{sec:hermstatespace}, \ref{sec:innerproduct}, Theorems \ref{prop:plaquette}, \ref{thm:hamilprop})
To $\Gr$-LW data, a triangulated surface $\Sigma$ with dual triangulation $\Gamma$, and a cohomology class $[\col]\in H^1(\Sigma,\Gr)$,
there exists a Hilbert space
$\HHH=\HHH(\Gamma,\col)$ equipped with a pseudo-Hermitian Hamiltonian   $H  \colon \HHH \rightarrow \HHH$.  This Hamiltonian is gapped (so the there is a finite energy gap between each excited state), local, and exactly solvable.
\end{theorem_noname}

\begin{theorem_noname} (Theorems \ref{thm:sphtodata}, \ref{thm:ground=TV})
A relative Hermitian $\Gr$-spherical category gives rise to $\Gr$-LW data.
Furthermore, the modified Hermitian TV invariant for a surface $\Sigma$ associated to $\Gr$ is isomorphic to the ground state of the Hamiltonian $H$ from the previous theorem.
\end{theorem_noname}

The original construction from \cite{GLPS2} based on semi-restricted quantum $\mf{sl}_2$ relied upon a forgetful functor between the unrolled quantum group and the semi-restricted quantum group for $\mathfrak{sl}_2$.  The auxiliary role of the unrolled quantum group is generalized here, allowing us to utilize data from a relative modular category to produce data needed to define Hermitian structures on relative spherical categories.    Indeed, in all of the examples considered above, we need this auxiliary braided category to define the   Hermitian structure.

\begin{theorem_noname} (Theorem \ref{T:relativeHer-premodularIsSpherical})
Under certain conditions, a relative Hermitian $\Gr$-modular category gives rise to a relative Hermitian $\Gr$-spherical category.
\end{theorem_noname}

In this paper we show this theorem applies to a large set of examples, including the categories of modules over the unrolled quantum groups associated to finite-dimensional simple complex Lie algebras and $\mathfrak{sl}(m|n)$.  In these examples the modules used in the construction only exist when the quantum parameter is a root of unity.  In a different direction, one can consider deformed modules over quantum
$\mathfrak{sl}(m|n)$ called perturbative modules.
These modules exist before specializing the quantum parameter to a root of unity and they form a graded category which is not semisimple where a generic module has vanishing quantum dimension.  In \cite{GPA,GPP} it is shown that the one can define an m-trace on these perturbative modules and  specialize $q$ to a root of unity; this forces some modules to have zero m-trace.  In the case of $\mathfrak{sl}(2|1)$,
 then taking the quotient by \emph{negligible morphisms} of modules with zero m-trace, it is shown that the resulting category is a relative $\Gr$-spherical category
 (see Theorem 5.6.1 of \cite{GPA}).  Thus, the category leads to a modified Turaev-Viro invariant of 3-manifolds with additional structures,
(in particular cohomology classes on surfaces and cobordisms, see \cite{GP3}).
 In this paper we show that the corresponding Levin-Wen Hamiltonian is pseudo-Hermitian.    The techniques of this paper were developed with this example in mind.

Let us say a few words about why we think this example is interesting.  Most of the examples of this paper are based on modules which only appear after the quantum parameter is specialized to a root of unity,
i.e.\ when the quantum parameter is specialized to a root of unity, new central elements of the algebra appear which allow for new families of simple modules.
 Instead the perturbative modules are deformations of classical $\mathfrak{sl}(2|1)$-modules specialized to a root of unity.  A theory based on deformations of classic modules should have connections with constructions in mathematical physics:  Mikhaylov and Witten in \cite{MWi} consider supergroup Chern-Simons (CS) theories (also see \cite{AGPS, Creutzig:2011cu, GaiottoWitten-Janus, Gotz:2006qp, Mik2015, KS1991, Quella:2007hr, RS1994}).
The motivating example of  \cite{MWi} is a path integral related to classical modules over $\mathfrak{gl}(1|1)$, (see \cite{Mik2015}).  As in the case of Witten's path integral interpretation of the Jones polynomial \cite{Witten}, in \cite{GY} it is shown that a De Renzi-Blanchet-Costantino-Geer-Patureau-Mirand style TQFT associated to the category of perturbative quantum $\mathfrak{gl}(1|1)$-modules recovers all the computable properties of the path integral given in \cite{Mik2015}.  However, it is not clear how to define a path integral for higher rank Lie superalgebras (in particular see Section 3.2 of \cite{MWi} for many questions about supergroup CS theories).  The results of \cite{GY} suggest that a solution to Conjecture \ref{c:bb=qdim} would lead to a family of Hermitian modified TV invariants which should correspond to not yet constructed higher rank supergroup CS path integrals.
 Thus, since Conjecture \ref{c:bb=qdim} is true for $\mathfrak{sl}(2|1)$, Theorem \ref{T:sl21ConjTrue} gives a pseudo-Hermitian relative LW-model associated to $\mathfrak{sl}(2|1)$ which should correspond to a higher rank supergroup CS path integral and therefore have interesting mathematical physics applications.  In particular,
this should be a combinatorial version of a (super) CS theory with
a non-standard continuous gauge group.  Such generalizations of CS theory have recently drawn much interest, for example see \cite{ABHH, CTGG2021, Dim, Dimofte2016PerturbativeAN,  DGLZ, Gukov2016ResurgenceIC}
and the references within.

\subsection{Outline}
We begin in Section \ref{sec:relativeLW} with the definition of a relative $\Gr$-LW-system.  Given this data and a triangulated, compact, connected, oriented surface, we define a Hilbert space and a Hamiltonian acting on this state space.
Section \ref{sec:sespiv} contains a definition of sesquilinear pivotal categories.  Hopf superalgebras equipped with a dagger map provide examples of such categories.
In Section \ref{sec:relhermmodsph} we review the notion of a modified trace and define relative Hermitian modular and relative Hermitian spherical categories.  We also give a construction on how to obtain a relative Hermitian spherical category from a relative Hermitian modular category.
Section \ref{sec:relHermtodata} contains a result showing that a relative Hermitian-spherical category leads to a relative $\Gr$-LW-system.
%
In Section \ref{sec:hermribbon} we show how to obtain a Hermitian WRT type TQFT.
We conclude with Section \ref{sec:examples} where we show how to obtain relative Hermitian-spherical categories from quantum groups for finite-dimensional simple complex Lie algebras and $\mathfrak{sl}(m|n)$.

\subsection{Acknowledgements}
N.G.\ is partially
supported by NSF grants DMS-1664387 and DMS-2104497.  A.D.L.\ is
partially supported by NSF grants DMS-1902092 and DMS-2200419, the Army Research Office
W911NF-20-1-0075, and the Simons Foundation collaboration grant on New Structures in Low-dimensional topology.  J.S.\ is partially supported by the NSF grant
DMS-1807161 and PSC CUNY Award 64012-00 52.

\section{ Relative $\Gr$-LW-system} \label{sec:relativeLW}
Let $\FK$ be a field equipped with an involutive automorphism
$  k \mapsto \wb k$.  We follow and generalize the system Hilbert space
described in \cite{GLPS2}.
\subsection{Basic Input} \label{subsec:basicinput}
The input for a relative $\Gr$-LW-system is a tuple
$(\Gr, \X, (I_g)_{g\in\Gr}, \delta, \mb, \md,N,\beta,\gamma)$, called  {\em $\Gr$-LW data}, where:
\begin{enumerate}
\item $(\Gr,+)$ is an abelian group.
\item $\X\subset \Gr$ is symmetric (i.e. $-\X=\X $)
  and small (i.e. for any finite subset $Y\subset\Gr$, $\Gr\not\subset Y+\X$).
\item $I_g$ is a set for a \emph{generic element} $g\in\Gr\setminus\X$. If $j\in I_g$, we call $g=|j|\in\Gr$ the {\em degree} of $j$.
\item Define $\A=\sqcup_{g\in\Gr\setminus\X}I_g$.  Then we assume $\A$ is equipped with an
  involution $*$ such that $I_g^*=I_{-g}$.

 \item The {\em fusion multiplicity} map $\delta:\A^3\to\N \maps (i,j,k) \mapsto \delta_{ijk}$ has dihedral symmetry: \[\delta_{ijk}=\delta_{jki}=\delta_{k^*j^*i^*},\]
and
  $\delta_{ijk}=0$ unless $|i|+|j|+|k|=0$.

 \item   The {\em modified dimension $\md :\A\to\FK^*$} and map $\mb :\A\to\FK^*$  satisfy
\[
\wb {\mb}= \mb,\quad  \wb {\md}= \md,\quad \mb(i^*)=\mb(i),\quad  \md(i^*)=\md(i).
\]
Furthermore, for $g,g_1, g_2\in \Gr\setminus \X$, with $g+g_1+g_2=0$, we have
 \begin{equation*} \label{eq:b}
   \bb(j)=\sum_{j_1\in \I_{g_1},\, j_2 \in \I_{g_2}}   \bb(j_1)\bb(j_2)\delta_{j^{\ast}j_1 j_2},
 \end{equation*}
 for all $j\in I_g$.

 \item There are {\em basic data} maps
 $\beta:\A\to\FK^*$ and $\gamma:\A^3\times \{1, \ldots, \delta(\A^3) \} \to\FK^*$ satisfying
 \[
 \wb {\gamma}= \gamma, \quad  \wb {\beta}= \beta, \quad \beta(i^*)=\beta(i)
 \]
 and for $1\le n\le\delta_{ijk}$, we have
\begin{equation}
    \label{eq:theta}
   \gamma_{i,j,k}^n\gamma_{k^*,j^*,i^*}^n\beta(i)\beta(j)\beta(k)=1 .
  \end{equation}

\item There are modified \emph{$6j$ symbols}
$
N:\A^6\times(\N^*)^4\to\FK
$
with the property that
  $N^{j_1 j_2 j_3}_{j_4 j_5 j_6}(^{a_1a_2}_{a_3a_4})=0$ unless
  $1\leq {a_1}\le\delta_{j_1 j_2 j_3^*}$,
  $1\leq{a_2}\le\delta_{j_3j_4j_5^*}$,
  ${1 \leq a_3}\le\delta_{j_5j_6^*j_1^*}$  and
 ${1 \leq a_4}\le\delta_{j_6j_4^*j_2^*}$.

\end{enumerate}
This data satisfies the following conditions
\begin{align}
  &N^{j_1 j_2 j_3}_{j_4 j_5 j_6}(^{a_1a_2}_{a_3a_4})= N^{j_2
    j_3^*j_1^* }_{j_5 j_6j_4 }(^{a_1a_3}_{a_4a_2})= N^{j_3 j_4
    j_5}_{j_6^* j_1 j_2^*}(^{a_2a_3}_{a_1a_4}) , \label{eq:symm} \\
  &\sum_{j,c_1,c_2,c_3} \md(j) N^{j_1 j_2 j_5}_{j_3 j_6 j}(^{a_1a_2}_{c_1c_2}) N^{j_1 j j_6}_{j_4 j_0 j_7}(^{c_1a_3}_{a_0c_3})
    N^{j_2 j_3 j}_{j_4 j_7 j_8}(^{c_2c_3}_{a_4a_5})=  \sum_{c_4}N^{j_5 j_3 j_6}_{j_4 j_0 j_8}(^{a_2a_3}_{c_4a_5})
    N^{j_1 j_2 j_5}_{j_8 j_0 j_7}(^{a_1c_4}_{a_0a_4}) , \label{eq:pent} \\
  &\sum_{n,a_3,a_4} \md(n) N^{i j p}_{l m n}(^{a_1a_2}_{a_3a_4}) N^{k j^{\ast} i}_{n m l}(^{a'_1a_3}_{a'_2a_4})
    = \frac{\delta_{k,p}\delta_{a_1,a'_1}{\delta}_{a_2,a'_2}}{\md(k)} ,
  \label{eq-ortho}\\
  & \wb{N^{j_1 j_2 j_3 }_{j_4j_5j_6}(^{a_1a_2}_{a_3a_4})}{=}
    N^{j_2^{\ast} j_1^{\ast} j_3^{\ast}}_{j_5j_4j_6}(^{a_1a_2}_{a_4a_3})
    \gamma^{a_1}_{j_1,j_2,j_3^{\ast}}\gamma^{a_2}_{j_3,j_4,j_5^{\ast}}\gamma^{a_3}_{j_1^{\ast},j_5,j_6^{\ast}}
    \gamma^{a_4}_{j_2^{\ast},j_6,j_4^{\ast}}\prod_{i=1}^6\beta(j_i) . \label{eq:6j}
    \end{align}

\subsection{Hermitian state space} \label{sec:hermstatespace}
Let $\Sigma$ be a compact, connected, oriented surface.
Let $\T$ be a triangulation of $\Sigma$ and $\Gamma$ be a finite
trivalent graph dual to $\T$.  Each vertex of the graph
$\Gamma$ acquires a cyclic ordering compatible with the orientation of
$\Sigma$.  A \emph{$\Gr$-coloring} of $\Gamma$ is 
a map $\col$ from the set of oriented edges
of $\Gamma$ to $\Gr$ such that the following hold.
\begin{enumerate}
\item $\col(-e)=-\col(e)^{}$ for any oriented edge $e$ of $\Gamma$, where
  $-e$ is $e$ with opposite orientation.
\item If $e_1, e_2, e_3$ are edges of a vertex $v$ of $\Gamma$ with a
  cyclic ordering compatible with the orientation of $\Sigma$, and each
  edge is oriented towards the vertex $v$, then
  $\col(e_1)+ \col(e_2)+ \col(e_3)=0$.
\end{enumerate}
The $\Gr$-colorings of $\Gamma$ form a group isomorphic via Poincar\'{e} duality to the group of $\Gr$-valued simplicial 1-cocycles on $\T$.  We denote by $[\col]\in H^1(\Sigma,\Gr)$ the associated cohomology class.
A $\Gr$-coloring of $\Gamma$ is admissible if
$\col(e)\in \Gr \setminus \X$ for any oriented edge $e$ of $\Gamma$.
A state of an admissible $\Gr$-coloring $\col$ is a map $\sigma$
assigning to every oriented edge $e$ of $\Gamma$ an element
$\sigma(e)\in I_{\col(e)}$ such that
$\sigma(-e)=\sigma(e)^{*}$   and to each vertex $v$ with incident edges $e_1, e_2, e_3$, the natural number $\sigma_0(v)$ where $0 \leq \sigma_0(v) \leq \delta_{\sigma(e_1),\sigma(e_2),\sigma(e_3)}$.
Denote by $\St(\col)$ the set of such
states.  For an admissible $\Gr$-coloring $\col$, we define the Hilbert
space $\HHH=\HHH(\Gamma,\col)$ as the span of all elements corresponding
to the states of $\col$:
\[
\HHH(\Gamma,\col)=\bigoplus_{\sigma\in\St(\col)}\C\ket{\Gamma,\sigma}.
\]
We write $\sigma^{\ast}$ for the state of the $\Gr$-coloring $-\Phi$ assigning $\sigma(e)^{\ast}$ to each oriented edge $e$ of $\Gamma$.

\subsection{Inner products} \label{sec:innerproduct}
We equip $\HHH=\HHH(\Gamma,\col)$ with a complete Hilbert space structure by defining a positive definite Hermitian inner product
\begin{equation}
 \la \cdot \mid \cdot \ra_+ \maps \HHH \otimes \HHH \to \C
\end{equation}
in which the states $\ket{\Gamma,\sigma}$ form an orthonormal basis:
\begin{equation}
\la \Gamma,\sigma \mid \Gamma, \sigma' \ra_+ =  \delta_{\sigma, \sigma'}.
\end{equation}
We will see that from the TQFT perspective, this is not the most natural inner product.  The most natural inner product can be obtained from this one utilizing a Hermitian operator $\eta \maps \HHH \to \HHH$.
The invertible operator $\eta$ is defined by
\begin{equation} \label{eq:def-eta}
  \eta \maps  \HHH(\Gamma,\col) \; \longrightarrow \; \HHH(\Gamma,\col) , \quad \quad
    \ket{\Gamma, \sigma}   \;\; \mapsto \;
   \frac{ \prod_{e\in\Gamma_1} \md(\sigma(e)) }{ \prod_{v\in \Gamma_0}\gamma(\sigma(v))\prod_{e\in \Gamma_1}\beta(\sigma(e)) }
   \ket{\Gamma, \sigma}.
\end{equation}
where $\sigma(v)=((j_1,j_2,j_3),\sigma_0(v)) \in I^3 \times \N$, where $1 \leq \sigma_0(v) \leq \delta_{j_1 j_2 j_3}$, $j_k=\sigma(e_k)$, and
$e_1,e_2,e_3$ are the $3$ ordered edges adjacent to $v$ oriented
toward $v$.
It is clear from the orthogonality of states $\ket{\Gamma,\sigma}$ and the fact that $\gamma(\sigma(v))$, $\md(\sigma(e))$ and $\beta(\sigma(e))$ are all fixed under the field involution, that
$\eta$ is Hermitian with respect to the form $\la\cdot \mid \cdot \ra_+$
\begin{equation} \label{eq:skew-eta}
 \la \psi \mid \eta \phi \ra_+ = \la \eta \psi \mid \phi \ra_+
\end{equation}
for all $\psi, \phi \in \cal{H}$.

Define a new inner product on $\HHH=\HHH(\Gamma,\col)$  via
\begin{equation} \label{eq:Aindef-inner}
\la \cdot \mid \cdot \ra := \la \cdot \mid   \eta^{-1}(\cdot) \ra_+, \quad \quad
\la \Gamma, \sigma \mid \Gamma, \sigma' \ra
 = \;  \frac{ \prod_{v\in \Gamma_0}\gamma(\sigma(v))\prod_{e\in \Gamma_1}\beta(\sigma(e)) }{ \prod_{e\in\Gamma_1} \md(\sigma(e)) }\ \cdot \delta_{\sigma, \sigma'}.
\end{equation}
It is not difficult to show that this new pairing is Hermitian (see \cite[Section II.C]{GLPS2}).
However, it is in general indefinite (see \cite[Sections II.C and V]{GLPS2}).
In what follows, for an operator $F$, we let $F^{\dagger}$ be the adjoint with respect to this indefinite form.

If a Hamiltonian $H$ is Hermitian with respect to the form $\la \cdot \mid \cdot \ra$, so that $\la \psi \mid H \phi \ra = \la H \psi \mid \phi \ra$, then the
Hermitian conjugate $H^{\sharp}$ of $H$ with respect to the form $\la \cdot \mid \cdot\ra_+$ satisfies
\begin{equation}
  H^{\sharp} = \eta^{-1} H \eta.
\end{equation}
Thus the  resulting Hamiltonian then becomes \emph{pseudo-Hermitian} with respect to the inner product $\la -,-\ra_+$.  See \cite[Section II.C]{GLPS2} for more details.

\subsection{Levin-Wen Hamiltonian and local operators}

Associated to each vertex $v \in \Gamma_0$, we have a \emph{vertex operator} $Q_v \maps \HHH(\Gamma,\Phi) \to \HHH(\Gamma,\Phi)$ that acts locally to impose the branching constraints from Section~\ref{subsec:basicinput}.
\begin{equation} \label{eq:vertexQ}
 Q_v\left| \hackcenter{\begin{tikzpicture}[   decoration={markings, mark=at position 0.6 with {\arrow{>}};}, scale =0.9]
    \draw[thick, blue, postaction={decorate}] (0,1) -- (.5,0);
    \draw[thick, blue, postaction={decorate}] (0,-1) -- (.5,0);
     \draw[thick, blue, postaction={decorate}] (1.5,0) -- (.5,0);
    \node at (-.2,.75) {$\scs j_2$};
    \node at (-.2,-.75) {$\scs j_1$};
   \node at (1.5,.25){$\scs j_3$};
      \node[draw, thick, fill=black!20,rounded corners=3pt,inner sep=2pt]   at (.5,0) {$c$};
\end{tikzpicture}}
\right\rangle
\;\; = \;\;
\left| \hackcenter{\begin{tikzpicture}[   decoration={markings, mark=at position 0.6 with {\arrow{>}};}, scale =0.9]
    \draw[thick, blue, postaction={decorate}] (0,1) -- (.5,0);
    \draw[thick, blue, postaction={decorate}] (0,-1) -- (.5,0);
     \draw[thick, blue, postaction={decorate}] (1.5,0) -- (.5,0);
    \node at (-.2,.75) {$\scs j_2$};
    \node at (-.2,-.75) {$\scs j_1$};
   \node at (1.5,.25){$\scs j_3$};
         \node[draw, thick, fill=black!20,rounded corners=3pt,inner sep=2pt]   at (.5,0) {$c$};
\end{tikzpicture}}
\right\rangle
\text{ if } 1 \leq c \leq \delta_{j_1 j_2j_3}, \text{ and } 0 \text{ otherwise}.
\end{equation}
It is straightforward to see that these operators are mutually commuting projectors, so that
\[
Q_v^2 = Q_v, \quad Q_v Q_{v'} = Q_{v'} Q_v \qquad \text{for $v,v' \in \Gamma_0$.}
\]
Furthermore, since $Q_v$ is a delta function, it is immediate that these operators are Hermitian with respect to the inner product $\la \cdot \mid \cdot \ra$ from \eqref{eq:Aindef-inner}.

  If
$g\in \Gr$ and $p$ is a plaquette in $\Sigma\setminus\Gamma$, then we write $g.\delta p$ for  the coloring
that sends oriented edges of $\delta p$ to $g$, their opposites
to $-g^{}$, and other edges to $0$. 

For $g\in \Gr\setminus\X$, let $\col$ be an admissible $\Gr$-coloring such that $\col+g.\delta p$ is also admissible.
Let $s\in I_{g}$  and define the  operator  $B_p^s:\HHH(\Gamma,\col)\to\HHH(\Gamma,\col+g.\delta p)$ which acts on the boundary edges of the plaquette
$p$  and is given on a plaquette with $n$ sides by:
\begin{equation}
B_p^s
\left| \hackcenter{\begin{tikzpicture}[     scale =0.9]
    \draw[thick, blue, directed=.65] (-.75,-.25) -- (0,-1);
    \draw[thick, blue, directed=.65]  (0,-1) --  (.75,-.25) ;
    \draw[thick, blue, directed=.75]  (-.75,-.25) --  (-1.35,-.45) ;
    \draw[thick, blue, directed=.65]  (.75,-.25) --  (1.35,-.45) ;
       \draw[thick, blue, directed=.65]  (0,-1) --  (0,-1.5) ;
        \draw[thick, blue, directed=.65]  (.75,-.25) --  (.75,.5) ;
         \draw[thick, blue, directed=.55]  (-.75,.5) --  (-.75,-.25) ;
         \draw[thick, blue, directed=.85]  (.75,.5) --  (1.25,.75) ;
           \draw[thick, blue, directed=.85]  (-.75,.5) --  (-1.35,.85) ;
    \node at (0,.5) {$\dots$};
     \node at (-.65,-.85) {$\scs j_1$};
      \node at (-1.1,.15) {$\scs j_n$};
      \node at (1.1,.15) {$\scs j_3$};
     \node at (.65,-.85) {$\scs j_2$};
     \node at (-1.3,-.65) {$\scs k_n$};
      \node at (1.3,-.65) {$\scs k_2$};
       \node at (-.80,.95) {$\scs k_{n-1}$};
      \node at (1.25,.95) {$\scs k_3$};
       \node at (-.3,-1.45) {$\scs k_1$};
       \node at (0,0) {$\scs p$};
                  \node[draw, thick, fill=black!20,rounded corners=3pt,inner sep=1pt]   at (-.75,-.25) {$\scs C_n$};
                  \node[draw, thick, fill=black!20,rounded corners=3pt,inner sep=1pt]   at (0,-.9)  {$\scs C_1$};
                  \node[draw, thick, fill=black!20,rounded corners=3pt,inner sep=1pt]   at(.7,-.2)  {$\scs C_2$};
                  \node[draw, thick, fill=black!20,rounded corners=3pt,inner sep=1pt]   at(.75,.5)  {$\scs C_3$};
                  \node[draw, thick, fill=black!20,rounded corners=3pt,inner sep=1pt]   at(-.7,.5)  {$\scs C_{n-1}$};
\end{tikzpicture}}
\right\rangle
\nonumber\\
 :=
\sum_{j_1',\ldots,j_n'}
\sum_{A_1, \ldots, A_n}
\sum_{C_1', \ldots, C_n'}
\prod_{i=1}^n
\md({j_i'})
N^{j_i' s j_i}_{j_{i+1}^{\ast} k_i j_{i+1}'^{\ast}} (^{A_i C_i}_{C_i' A_{i+1}})
\left| \hackcenter{\begin{tikzpicture}[     scale =0.9]
    \draw[thick, blue, directed=.55] (-.75,-.25) -- (0,-1);
    \draw[thick, blue, directed=.65]  (0,-1) --  (.75,-.25) ;
    \draw[thick, blue, directed=.75]  (-.75,-.25) --  (-1.35,-.45) ;
    \draw[thick, blue, directed=.65]  (.75,-.25) --  (1.35,-.45) ;
       \draw[thick, blue, directed=.65]  (0,-1) --  (0,-1.5) ;
        \draw[thick, blue, directed=.65]  (.75,-.25) --  (.75,.5) ;
         \draw[thick, blue, directed=.65]  (-.75,.5) --  (-.75,-.25) ;
         \draw[thick, blue, directed=.85]  (.75,.5) --  (1.25,.75) ;
           \draw[thick, blue, directed=.85]  (-.75,.5) --  (-1.35,.85) ;
    \node at (0,.5) {$\dots$};
     \node at (-.65,-.85) {$\scs j_1'$};
      \node at (-1.1,.05) {$\scs j_n'$};
      \node at (1.1,.1) {$\scs j_3'$};
     \node at (.65,-.85) {$\scs j_2'$};
     \node at (-1.3,-.65) {$\scs k_n$};
      \node at (1.3,-.65) {$\scs k_2$};
       \node at (-.80,.95) {$\scs k_{n-1}$};
      \node at (1.25,.95) {$\scs k_3$};
       \node at (-.3,-1.45) {$\scs k_1$};
       \node at (0,0) {$\scs p$};
                  \node[draw, thick, fill=black!20,rounded corners=3pt,inner sep=1pt]   at (-.75,-.25) {$\scs C_n'$};
                  \node[draw, thick, fill=black!20,rounded corners=3pt,inner sep=1pt]   at (0,-.9)  {$\scs C_1'$};
                  \node[draw, thick, fill=black!20,rounded corners=3pt,inner sep=1pt]   at(.7,-.2)  {$\scs C_2'$};
                  \node[draw, thick, fill=black!20,rounded corners=3pt,inner sep=1pt]   at(.75,.5)  {$\scs C_3'$};
                  \node[draw, thick, fill=black!20,rounded corners=3pt,inner sep=1pt]   at(-.7,.45)  {$\scs C_{n-1}'$};
\end{tikzpicture}}
\right\rangle .
\nonumber
\end{equation}
Note that the conditions on $6j$ symbols in Section \ref{subsec:basicinput} ensure that $\deg(j_i')= \deg(j_i) - \deg(s)$, otherwise the right hand side is zero in the above formula.

  Then we define $\Gr$-indexed {\em plaquette operators}
\begin{equation} \label{eq:Balpha}
  B_p^{g}=
  \sum_{s\in I_{g}}\bb(s)B_p^{s}:\HHH(\Gamma,\col)\to\HHH(\Gamma,\col+g.\delta p).
\end{equation}

\begin{proposition} \label{prop:plaqexp}
 If $g_1$, $g_2$,  and $g_1+g_2$ are generic, then
  $$B_p^{g_1}B_p^{g_2}=B_p^{g_1+g_2}.$$
  Moreover, if $g_1$ and $g_2$ are generic then $B_p^{g_1}B_p^{-g_1^{}}=B_p^{g_2} B_p^{-g_2^{}}$.
\end{proposition}

\begin{proof}
This is a straightforward generalization of \cite[Proposition 3.1]{GLPS2}.
\end{proof}

\begin{proposition} \label{1edgecomm}
Let $p$ and $p'$ be two plaquettes which share exactly one common edge $e_1$ in a state $\ket{\Gamma,\sigma}$ as in \eqref{adjplaq1}.
Then $B_{p'}^t B_{p}^s=B_{p}^s B_{p'}^t $.
\begin{equation} \label{adjplaq1}
\ket{\Gamma, \sigma} \;\; := \;\;
\hackcenter{\begin{tikzpicture}[  scale =0.9]
    \draw[thick, blue, directed=.55] (0,-.5) -- (0,.5);
    \draw[thick, blue, directed=.65] (-.75,-1) -- (0,-.5);
    \draw[thick, blue, directed=.65] (.75,-1) -- (0,-.5);
     \draw[thick, blue, directed=.55] (0,.5) -- (-.75,1);
     \draw[thick, blue, directed=.55] (0,.5) -- (.75,1);
      \draw[thick, blue, directed=.8] (-.75,1) -- (-.75,1.5);
      \draw[thick, blue, directed=.8] (.75,1) -- (.75,1.5);
       \draw[thick, blue, directed=.8] (-1.75,1) -- (-2,1.5);
      \draw[thick, blue, directed=.8] (1.75,1) -- (2,1.5);
     \draw[thick, blue, directed=.5] (-.75,-1.5) -- (-.75,-1);
     \draw[thick, blue,directed=.5] (.75,-1.5) -- (.75,-1);
     \draw[thick, blue, directed=.55] (.75,1) -- (1.75,1);
       \draw[thick, blue, directed=.75] (1.75,1) -- (2.25,.5);
      \draw[thick, blue, directed=.55] (1.75,-1) -- (.75,-1);
       \draw[thick, blue, directed=.55] (-1.75,-1) -- (-.75,-1);
       \draw[thick, blue, directed=.55] (-.75,1) -- (-1.75,1);
       \draw[thick, blue, directed=.75] (-1.75,1) -- (-2.25,.5);
      \node at (-2.25,0){$\ddots$};
     \node at (-1.25,0){$p$};
    \node at (1.25,0){$p'$};
   \node at (.3,.1){$\scs e_1$};
   \node at (-.15,-1.05){$\scs f_{a+1}$};
   \node at (.78,-.55){$\scs h_{b+1}$};
    \node at (1.45,-1.3){$\scs h_{b}$};
    \node at (-1.45,-1.3){$\scs f_a$};
     \node at (-.3,1.1){$\scs f_2$};
     \node at (-2.3,.9){$\scs f_4$};
      \node at (2.4,.9){$\scs h_4$};
   \node at (.25,1.1){$\scs h_2$};
   \node at (.9,-1.65){$\scs \ell_b$};
   \node at (1,1.65){$\scs \ell_2$};
   \node at (2.25,1.5){$\scs \ell_3$};
   \node at (-.9,-1.65){$\scs k_a$};
   \node at (-.95,1.65){$\scs k_2$};    \node at (-2.2,1.65){$\scs k_3$};
    \node at (-1.4,1.35){$\scs f_3$};
      \node at (1.4,1.35){$\scs h_3$};
      \node at (2.25,-.25){$\vdots $};
    \node[draw, thick, fill=black!20,rounded corners=3pt,inner sep=1pt]   at (0,-.5) {$\scs C_2$};
    \node[draw, thick, fill=black!20,rounded corners=3pt,inner sep=1pt]   at (0,.5) {$\scs C_1$};
        \node[draw, thick, fill=black!20,rounded corners=3pt,inner sep=1pt]   at (-.75,1) {$\scs B_2$};
        \node[draw, thick, fill=black!20,rounded corners=3pt,inner sep=1pt]   at (.75,1) {$\scs D_2$};
         \node[draw, thick, fill=black!20,rounded corners=3pt,inner sep=1pt]   at (1.75,1) {$\scs D_3$};
                 \node[draw, thick, fill=black!20,rounded corners=3pt,inner sep=1pt]   at (-1.75,1) {$\scs B_3$};
                             \node[draw, thick, fill=black!20,rounded corners=3pt,inner sep=1pt]   at (-.75,-1) {$\scs B_a$};
                             \node[draw, thick, fill=black!20,rounded corners=3pt,inner sep=1pt]   at (.75,-1) {$\scs D_b$};
                      \end{tikzpicture}}
\end{equation}
\end{proposition}

\begin{proof}
This is a straightforward generalization of \cite[Proposition 3.2]{GLPS2}.
\end{proof}

The next result follows easily from the previous proposition and is a generalization of \cite[Corollary III.3]{GLPS2}.
\begin{corollary} \label{corplaqcomm}
Let $p$ and $p'$ be two plaquettes.
Then $B_{p'}^t B_{p}^s=B_{p}^s B_{p'}^t $.
\end{corollary}

\begin{proof}
If $p$ and $p'$ are identical, or if they share no edges, the result is obvious.
The case that they share exactly one edge is Proposition \ref{1edgecomm}.
\end{proof}

\begin{proposition} \label{plaqhermprop}
The Hermitian adjoint of $B_p^s$ is $B_p^{s^{\ast}}$.  In particular,
\[
\left(B_p^{g}\right)^{\dagger} = B_p^{-g} .
\]
\end{proposition}

\begin{proof}
This is a straightforward generalization of \cite[Proposition III.4]{GLPS2}.
\end{proof}

Define the \emph{plaquette operator}
\begin{equation}
B_p:= B_p^{g}B_p^{-g^{}}:\HHH(\Gamma,\col)\to\HHH(\Gamma,\col).
\end{equation}
 A priori, this operator depends on ${g}$, but Proposition \ref{prop:plaqexp} and the smallness of the singular set $\X$ guarantees that the operator is independent of $g$.


The next result follows from the facts above.  See the arguments in \cite[Theorem III.6]{GLPS2} for more details.

\begin{theorem} \label{prop:plaquette}
The plaquette operators have the following properties.
\begin{enumerate}
    \item The plaquette operators are projectors $B_p^2 = B_p$.
    \item The plaquette operators are mutually commuting:
\[
 B_p B_{p'} = B_{p'} B_p .
\]

    \item The operators $B_p$ are Hermitian with respect to the indefinite norm \eqref{eq:Aindef-inner}.
\end{enumerate}
\end{theorem}

It is also straightforward to see that the plaquette operators commute with the vertex operators from \eqref{eq:vertexQ}, so that we can define a {\em local commuting projector Hamiltonian}.
\begin{equation} \label{eq:Hamiltonian}
H =   - \sum_{p} B_p-\sum_{v} Q_v.
\end{equation}
We can shift this Hamiltonian so that the ground state has zero energy.
\begin{equation} \label{eq:Hamiltonian0}
H =   \sum_{p} (1-B_p)+ \sum_{v} (1-Q_v).
\end{equation}

\begin{theorem} \label{thm:hamilprop}
The Hamiltonian from \eqref{eq:Hamiltonian0} has the following properties.
\begin{enumerate}
  \item  The Hamiltonian is Hermitian with respect to the indefinite inner product.

  \item The Hamiltonian is gapped and local.

  \item The spectrum of $H$ is $\Z_{\geq 0}$, with ground state  corresponding to the simultaneous +1-eigenspace of the operators $B_p$ and $Q_v$ over all vertices $v$ and plaquettes $p$.
\end{enumerate}
\end{theorem}

\begin{proof}
The first claim follows from  Theorem~\ref{prop:plaquette}.  Since the Hamiltonian is constructed from mutually commuting projectors  acting locally on the graph $\Gamma$, the second and third claims follow.
\end{proof}

\section{ Sesquilinear pivotal categories} \label{sec:sespiv}

\subsection{Definitions} \label{sec:Def-sesqi-pivotal}
A
\emph{tensor $\FK$-category} is a tensor category $\cat$ such that its
hom-sets are left $\FK$-modules, the composition and tensor product of
morphisms are $\FK$-bilinear, and the canonical $\FK$-algebra map
$\FK \to \End_\cat(\unit), k \mapsto k \, \Id_\unit$ is an isomorphism
(where $\unit$ is the unit object).
\begin{definition}\label{def:GenConj}
  Let $\cat$ be a tensor $\FK$-category.  A   not necessarily  \emph{involutive conjugation} assigns for each $V,W\in \cat$ a bijection
  $f\mapsto f^{\dagger}$ between $\Hom(V,W)$ and $\Hom(W,V)$
  such that \
\begin{equation}
  (\mu f+\lambda g)^{\dagger}=\wb\mu f^{\dagger}+\wb\lambda g^{\dagger},\quad (f\otimes g)^{\dagger}  = f^{\dagger} \otimes g^{\dagger} , \quad
  (f\circ g)^{\dagger}  = g^{\dagger}  \circ f^{\dagger}.
\end{equation}
\end{definition}
These relations imply $ \Id_V^{\dagger} = \Id_V$.

\begin{definition} \label{pivdef} A tensor category is
\emph{pivotal} if it has dual objects and duality morphisms
$$\coev_{V} : \:\:
\unit \rightarrow V\otimes V^{*} , \quad \ev_{V}: \:\: V^*\otimes V\rightarrow
\unit , \quad \tcoev_V: \:\: \unit \rightarrow V^{*}\otimes V\quad {\rm {and}}
\quad \tev_V: \:\: V\otimes V^*\rightarrow \unit$$ which satisfy compatibility
conditions (see for example \cite{BW, GKP2}).  In particular, it has natural isomorphisms
$$\{\phi_V=(\tev_V \otimes \,\id_{V^{**}})(\id_V\otimes \coev_{V^*}):V\to V^{**}\}$$
called the pivotal structure.
\end{definition}
\begin{definition} \label{sespivdef}
 A {\em sesquilinear pivotal category} is a pivotal category $\cat$ equipped with a not necessarily involutive conjugation
 satisfying:
 \begin{equation}\label{eq:Hermitian+dual}
 \coev_V^{\dagger} = \tev_V,
\quad
 \ev_V^{\dagger} =
  \tcoev_V \quad\tcoev_V^{\dagger} = \ev_V,
\quad
 \tev_V^{\dagger} =
  \coev_V
\end{equation}
for any object $V$ of $\cat$.
These identities imply that $\phi_V^{\dagger}=\phi_V^{-1}$.
\end{definition}
Recall that the pivotal category $\cat$ is ribbon if it has a braiding
$\bp{c_{V,W}}_{V,W\in\cat}$ which satifies for any object $V\in\cat$,
$$(\Id\otimes\tev_V)(c_{V,V}\otimes\Id)(\Id\otimes\coev_V)=
(\ev_V\otimes\Id)(\Id\otimes c_{V,V})(\tcoev_V\otimes\Id).$$
This endomorphism of $V$ is called the twist $\theta_V$.
\begin{definition} \label{sesribdef} A {\em sesquilinear ribbon
    category} is a sesquilinear pivotal category $\cat$ which is
  ribbon with a braiding satisfying:
 \begin{equation}\label{eq:Hermitian+braiding}
 c_{V,W}^{\dagger} = c_{V,W}^{-1} ,
\end{equation}
for any objects $V,W$ of $\cat$.
This identity implies that $\theta_V^{\dagger}=\theta_V^{-1}$.
\end{definition}

\begin{definition} \label{sesribdef} A sesquilinear pivotal (or ribbon) category is called a {\em Hermitian pivotal (or ribbon) category} if for any morphism $f$ in $\cat$, one has
 \begin{equation}\label{eq:Hermitian}
 (f^{\dagger})^{\dagger} = f.
\end{equation}
\end{definition}

\subsection{Sesquilinear super vector spaces.} Let us call a
sesquilinear super vector space a finite-dimensional super $\FK$-vector
space $V$ equipped with a non-degenerate graded sesquilinear
form $(\cdot\mid \cdot)_V$, such that for any $v_1,v_2,v_3\in V$ and
any $\lambda\in\FK$,
\[
  (\mu v_1+\lambda v_2 \mid v_3)=\wb\mu (v_1 \mid v_3)+\wb\lambda (v_2 \mid v_3)\text{ and }
  (v_1 \mid \mu v_2+\lambda v_3)=\mu (v_1 \mid v_2)+\lambda (v_1 \mid v_3),
\]
and $(v_1 \mid v_2)=0$ if $v_1,v_2$ are homogeneous with opposite
parity.

Let $\dag\FK$-$\SVect$ be the category of sesquilinear super vector
spaces with even $\FK$-linear maps.  Let $V$ and $W$
be objects in $\dag\FK$-$\SVect$ with sesquilinear forms
$(\cdot\mid \cdot)_V$ and $(\cdot\mid \cdot)_W$, respectively.  Define
$(\cdot \mid \cdot)_p$ on $V\otimes W$ by
\begin{equation}\label{E:DefForm-p}
(v\otimes w \mid v'\otimes w')_p=(-1)^{|v'||w|}(v \mid v')_V(w \mid w')_W
\end{equation}
for $v,v'\in V$ and $w,w'\in W$, and $|x|$ means the degree of an element $x$.  Then $V\otimes W$ is an object of $\dag\FK$-$\SVect$ with the sesquilinear form
$(\cdot\mid \cdot)_{p}$.
Any basis $\{e_i\}_i$ of $V$ has a right sesquilinear dual basis
$\{e'_i\}_i$ such that $(e_i\mid e'_j)_V=\delta_{ij}$.  Then the dual vector
space $V^*=\Hom_{\FK}(V,\FK)$ is equipped with the sesquilinear form
given by
\begin{equation}
  \label{eq:dagdual}
  (\varphi\mid \psi)_{V^*}=\sum_i\wb{\varphi(e_i)}\psi(e'_i).
\end{equation}
This does not depend on the choice of $\{e_i\}$.
The category $\dag\FK$-$\SVect$ is pivotal with duality morphisms given by
 \begin{align*}
  \coev_V :& \FK \rightarrow V\otimes V^{*}, \quad 1 \mapsto \sum
  v_i\otimes v_i^*,  &
  \ev_V: & V^*\otimes V\rightarrow \FK, \quad
  f\otimes v \mapsto f(v),\\
   \tcoev_V :& \FK \rightarrow V^*\otimes V, \quad 1 \mapsto \sum
 (-1)^{|v_i|} v_i^*\otimes v_i,  &
  \tev_V: & V\otimes V^*\rightarrow \FK, \quad
  v\otimes f \mapsto (-1)^{|f||v|}f(v).
\end{align*}
 The dagger is the right adjoint for the sesquilinear form. That is, given $f:V\to W$, then $\dag f:W\to V$ is the unique map defined by
\begin{equation}
  \label{eq:dagf}
  (f(v)\mid w)_W=(v\mid  \dag f(w))_V
\end{equation}
for all  $(v,w)\in V\times W$.
Equipped with these maps, we have the following result.
\begin{proposition}
  The category $\dag\FK$-$\SVect$ is a sesquilinear pivotal category.
\end{proposition}
We note that if $V$,$W$ are sesquilinear spaces,
$\Hom_{\dag\FK\hbox{-}\SVect}(V,W)=\Hom_{\FK}(V,W)_{\wb0}$ is the space
of even linear maps.  The dagger admits a natural extension for all
homogeneous linear maps $f$ by
\begin{equation}
  \label{eq:sdagf}
  (f(v)\mid w)_W=(-1)^{\de f\de v}(v\mid  \dag f(w))_V .
\end{equation}
With this definition, one easily check that the adjoint of a composition is given by
\begin{equation}
  \label{eq:sdag-prod}
  \dag{\bp{fg}}=(-1)^{\de f\de g}\dag g\dag f .
\end{equation}
We conclude with some considerations about symmetry.
\begin{definition}
A sesquilinear form on $W$ is called {\em super Hermitian} if it has super Hermitian
symmetry: for any $v_1,v_2\in W$, $(v_2|v_1)=(-1)^{|v_1| \cdot |v_2|}\wb{(v_1|v_2)}$. It is
{\em Hermitian} if for any $v_1,v_2\in W$, $(v_2|v_1)=\wb{(v_1|v_2)}$.
\end{definition}
\begin{lemma}\label{L:invol-herm}
  Let $V$ and $W$ be sesquilinear super vector spaces.
  \begin{enumerate}
  \item If the forms on $V$ and $W$ are super Hermitian, then the dagger is
    involutive, meaning that for any homogeneous linear map
    $f:V\to W$, $f^{\dagger\dagger}=f$.
  \item If the forms on $V$ and $W$ are Hermitian, then the dagger is
    superinvolutive, meaning that for any homogeneous linear map
    $f:V\to W$, $f^{\dagger\dagger}=(-1)^{\de f}f$.
  \end{enumerate}
\end{lemma}
\begin{proof}
  Let $s=\wb0$ if the form is Hermitian and $s=\wb 1$ if the form
  is super Hermitian, so that $(x|y)=(-1)^{s\de x}\wb{(y|x)}$. For
  $f:V\to W$, one has
  $$\forall (v,w)\in V\times W,\,(f(v)|w)_W=(-1)^{\de f\de v}(v|\dag f(w))_V=
  (-1)^{\de f\de v}(-1)^{s\de v}\wb{(\dag f(w)|v)}_V$$
  $$=(-1)^{\de f\de v}(-1)^{s\de v}(-1)^{\de f\de w}\wb{(w|f^{\dagger\dagger}(v))_W}
  =(-1)^{\de f\de v}(-1)^{s\de v}(-1)^{\de f\de w}(-1)^{s\de
    w}(f^{\dagger\dagger}(v)|w)_W.$$
    But if this scalar is non-zero,
  then $\de f+\de v+\de w=\wb0$. Then
  $(-1)^{\de f\de w}(-1)^{\de f\de v}=(-1)^{\de f}$ and
  $(-1)^{s\de v}(-1)^{s\de w}=(-1)^{s\de f}$.  So overall we get $f^{\dagger\dagger}=\left\{
    \begin{array}{lr}
      (-1)^{\de f}f&\text{ if }s=\wb 0\\
      f&\text{ if }s=\wb 1\\
    \end{array}\right.$.
\end{proof}

\subsection{Hopf superalgebras}\label{SS:HopfSuAlg}

Let $H$ be a ribbon Hopf superalgebra over $\FK$ with multiplication
$m : H \otimes H \to H$, unit $\eta : \FK \to H$, coproduct
$\Delta : H \to H \otimes H$, counit $\epsilon : H \to \FK$ and
antipode $S : H \to H$.  Recall that in a Hopf superalgebra the antipode satisfies
\begin{equation} \label{eq:superanti}
S(xy) = 
{(-1)^{|x||y|}}S(y)S(x).
\end{equation}
We denote by $R,\theta$ and $g$ the R-matrix, twist and pivot of $H$,
respectively.  All these structure maps and elements are assumed to be of even parity.  Strictly speaking, many of the examples we consider in
Section~\ref{sec:examples} are not ribbon Hopf algebras, but rather
Hopf algebras equipped with operators $R$, $\theta$, $g$ acting on the
category $H$-mod of finite-dimensional $H$-modules and satisfying the
requisite relations as operators.  Put another way, these are not
elements of $H$, but rather they live in an appropriate completion of
$H$, giving well-defined operators on the category of modules.  In
what follows, by abuse of notation we sometimes refer to {\em
  elements} in this more general situation.

Assume $H$ is equipped with \emph{anti-superautomorphism} which is a bijective map $\dagger:H\to H$
that is antilinear, anti-superalgebra which is also a coalgebra
anti-supermorphism\footnote{In the non super case, the convention of the
  coproduct corresponds to the notion of a twisted star Hopf algebra
  in \cite{CGT}.}.
%
Thus we assume $\text{for any } a \in \FK \text{ and } x,y \in H,$
\begin{equation}
  \label{eq:dag2}
 \dagp{ax}=\bar{a}\dag{x}, \quad \dagp{xy}=\sig x y\dag{y}\dag{x}
  \text{ and } \Delta(\dag{x}) =(\dagger\otimes\dagger)(\tau(\Delta x)),
\end{equation}
where $\tau:H\otimes H\to H\otimes H$ is the super flip map given by
$x\otimes y \mapsto (-1)^{|x|\cdot |y|}y\otimes x$.
We let
$R_{21}=\tau(R)$.  Finally, we also assume that
\begin{equation}
  \label{eq:Rdag2}
  (\dagger\otimes\dagger)\bp{R}=\tau(R^{-1})\text{ and }g^\dagger=g^{-1}.
 \end{equation}

We can deduce that for any $x\in H$ we have
$\epsilon( \dag x)=\epsilon(x)$,  $S(\dag x)=S(x)^{\dagger}$, $\dag{\theta}=\theta^{-1}$ and $\dag1=1\in H$.  For the second fact, see \cite{Sch}.

Recall that the twist on an $H$-module $V$ is the
endomorphism $\theta_V$ given by the action of $\theta^{-1}$.
The dual of an $H$-module $V$ is the superspace $V^{\ast}$ equipped
with the following action
\begin{equation} \label{eq:superdual}
(x \cdot \phi) (v) := 
{(-1)^{|x| \cdot |\phi|}} \phi(S(x)v) .
\end{equation}
The category of finite-dimensional $H$-modules with even
morphisms is known to be a pivotal $\FK$-tensor category (see for
example \cite[p. 163]{CP}) where the duality morphisms for an
$H$-module $V$ are given by
 \begin{align}\label{E:PivotStrH}
  \coev_V :& \FK \rightarrow V\otimes V^{*}, \quad
  1 \mapsto \sum
  v_i\otimes v_i^*,  &
  \ev_V: & V^*\otimes V\rightarrow \FK, \quad
  f\otimes v \mapsto f(v), \notag \\
   \tcoev_V :& \FK \rightarrow V^*\otimes V, \quad
   1 \mapsto \sum
 (-1)^{|v_i|}  v_i^*\otimes g^{-1}v_i,  &
  \tev_V: & V\otimes V^*\rightarrow \FK,  \quad
  v\otimes f \mapsto (-1)^{|f||v|}f(gv).
\end{align}

A \emph{sesquilinear $H$-module} is a finite-dimensional $H$-module $V$ equipped with a non-degenerate
sesquilinear form $(.\mid.)_V$ compatible with $\dagger$:
\begin{equation}\label{E:sesquilinearComp}
  (uv\mid v')_V= 
  {(-1)^{|u| \cdot |v|}}(v\mid u^\dagger v')_V
\end{equation}
where $u\in H$ and $v,v'\in V$.
Define $H^\dagger$-mod to be the category of sesquilinear $H$-modules
with even $H$-module morphisms.
The dual of a sesquilinear $H$-module inherits the sesquilinear form
given in \eqref{eq:dagdual} which can be shown to be compatible with $\dagger$.

Let $V,W$ be modules in $H^\dagger$-mod. Define $(\cdot \mid \cdot)_p$
on $V\otimes W$ by
\begin{equation}\label{E:DefForm-p2}
(v\otimes w \mid v'\otimes w')_p= 
{(-1)^{|v'||w|}}(v \mid v')_V(w \mid w')_W
\end{equation}
for $v,v'\in V$ and $w,w'\in W$.
The tensor product $V\otimes W$ can be given two different sesquilinear structures.  The first one is canonical and is given by defining
\begin{equation}
  \label{eq:sesq-tens}
 (v_1\otimes w_1\mid v_2\otimes w_2)^R_{V \otimes W}=  (v_1\otimes w_1\mid R(v_2\otimes w_2))_p
\end{equation}
where $(\cdot \mid \cdot)_p$ is defined in \eqref{E:DefForm-p2}.
This sesquilinear form is not compatible with the pivotal structure given in \eqref{E:PivotStrH}.

The second one depends of the choice of a half twist which we now
recall from \cite{GLPS1}.  A {\em half twist} $\dt$ for $H$-mod is a natural
isomorphism of the identity functor whose square is the twist such that
$\dt_{V^*}=(\dt_V)^*$.  We now assume $H$-mod has a fixed choice of half twist. By
 \cite[Proposition 4.12]{GLPS1}, a half twist exists if $\FK$ is an
algebraically closed field of characteristic 0.  For objects $W_1,W_2$
of $H$-mod, define the isomorphism
$\XX: W_1\otimes W_2 \to W_2\otimes W_1 $ by
\begin{equation} \label{defofX}
\XX_{W_1,W_2}=\bp{\dt_{W_2\otimes W_1}}^{-1}c_{W_1,W_2}(\dt_{W_1}\otimes \dt_{W_2}).
\end{equation}

 The forms given in \eqref{E:DefForm-p2} and  \eqref{eq:sesq-tens}  are not compatible with $\dagger$ but we have the following modification.

\begin{proposition}\label{P:tensor} Let $W_1,W_2$ be objects of $H^\dagger$-mod.  Then $W_1 \otimes W_2$ is a
  sesquilinear $H$-module with sesquilinear structure given by
   $$(v| v') = (v| \tau(\XX_{W_1 \otimes W_2} v'))_p$$
   where $\tau:W_2 \otimes W_1 \to W_1 \otimes W_2$ is the super flip
   map given by $\tau(v\otimes w)=(-1)^{|w||v|}w\otimes v$.
   \\
   This defines a monoidal structure on $H^\dagger$-mod.
\end{proposition}
\begin{proof}
This proof is reproduced from \cite{GLPS1}, but we now include all the relevant signs arising from super considerations.
Let $u\in H$ and write $\Delta(u)=u_1\otimes u_2$. Consider
  $v,v'\in W_1 \otimes W_2$ with $v=v_1\otimes v_2$ , $v'=v_1' \otimes v_2'$, and
  $v''=\XX_{W_1 \otimes W_2} (v')=v_2''\otimes v_1''\in W_2 \otimes
  W_1$.   Note that we have omitted all summation symbols and will continue to do so throughout the course of the proof.
  Then we have
\begin{align*}
  (uv\mid v')
  &=(uv\mid \tau(\XX_{W_1 \otimes W_2} v'))_p
    =(uv\mid \tau (v''))_p
  \\
  & =\sig{u_2}{v_1}\sig{v''_1}{v''_2}(u_1v_1 \otimes u_2v_2\mid   v_1'' \otimes v_2'' )_p
  \\
  & = \sig{u_2}{v_1}\sig{v''_1}{v''_2}\sig{u_2v_2}{v''_1}
    (u_1v_1 \mid v''_1)_{W_1}(u_2v_2\mid v''_2)_{W_2}
  \\
  &
    = \sig{u_2}{v_1}\sig{v''_1}{v''_2}\sig{u_2v_2}{v''_1}
    (-1)^{|u_1| |v_1| + |u_2| |v_2|}
    (v_1\mid \dag{u_1}v''_1)_{W_1}(v_2\mid \dag{u_2}v''_2)_{W_2}
  \\
  & = \sig{u}{v}\sig{u_1}{v_2}\sig{v''_1}{v''_2}\sig{u_2v_2}{v''_1}
    (v_1\mid \dag{u_1}v''_1)_{W_1}(v_2\mid \dag{u_2}v''_2)_{W_2}
  \\
  & = \sig{u}{v}\sig{u_1}{v_2}\sig{v''_1}{v''_2}\sig{u_2v_2}{v''_1} \sig{\dag{u_1}v''_1}{v_2}
    (v_1\otimes v_2\mid \dag{u_1}v''_1\otimes \dag{u_2}v''_2)_p
  \\
  & = \sig{u}{v}(-1)^{|v''_1|(|v''_2|+|u_2|)}
    (v_1\otimes v_2\mid \dag{u_1}v''_1\otimes \dag{u_2}v''_2)_p
  \\
  & = \sig{u}{v}(-1)^{|u_1|(|v''_2|+|u_2|)}
    (v_1\otimes v_2\mid \tau(\dag{u_2}v''_2\otimes \dag{u_1}v''_1)_p
  \\
  & = \sig{u}{v}(v_1\otimes v_2\mid
    \tau(\sig{u_1}{u_2}(\dag{u_2}\otimes \dag{u_1})(v''_2\otimes v''_1))_p
  \\
  & = \sig{u}{v}
    (v_1\otimes v_2\mid (\tau(\Delta(\dag{u})\XX_{W_1 \otimes W_2} (v'_1\otimes v'_2)))_p
  \\
  & = \sig{u}{v}(v\mid \tau\circ\XX_{W_1 \otimes W_2} (\dag{u}v'))_p
  \\
   &= \sig{u}{v}(v\mid \dag{u}v')
  \end{align*}
  where we used
  $\Delta(\dag{u})=(-1)^{|u_2||u_1|}\dag{u_2}\otimes\dag{u_1}$.  Hence we have defined a
  compatible sesquilinear form on $W_1 \otimes W_2$.

  In order to check that this defines a tensor product on
  $H^\dagger$-mod, we need to check the unital and associativity
  axioms, which in turn follow from the properties of
  $\XX$.  The unital axiom follows from the fact that
  $$\bp{W\stackrel{\sim}\to W\otimes\unit\stackrel{\XX}\to\unit\otimes W\stackrel{\sim}\to W}
  =\Id_W=\bp{W\stackrel{\sim}\to \unit\otimes W\stackrel{\XX}\to W\otimes\unit\stackrel{\sim}\to W},$$
  which is a  direct consequence of the definition of $\XX$.
  The associativity
  axiom follows from the equality
  $$\XX_{V'\otimes V,V''}(\XX_{V,V'}\otimes\Id_{V''})
  =\XX_{V,V''\otimes V'}(\Id_{V}\otimes\XX_{V',V''}): V\otimes V'\otimes
  V''\to V''\otimes V'\otimes V,$$
  which is proved in \cite[Lemma 4.13]{GLPS1}.
\end{proof}

\begin{theorem}\label{T:SequilinearPivotal}
  $H^\dagger$-mod equipped with a choice of half twist $\dt$ is a sequilinear pivotal $\FK$-category with a not necessarily involutive conjugation determined by:
\[
\forall f\in\Hom_{H^\dagger\text{-}\mod}(V,W),\forall v\in V,\forall w\in W, \quad (f(v)|w)_W=(v|f^\dagger(w))_V.
\]
  Furthermore, $\dag{{c_{V,W}}}={c_{V,W}}^{-1}$ where $c_{V,W}$ is the braiding defined by $c_{V,W}(v\otimes w)=\tau (R(v\otimes w))$\footnote{If in addition, $(f^{\dagger})^{\dagger}=f$  then these conditions imply $H^\dagger$-mod is a Hermitian ribbon category in the sense of \cite[Definition 3.1]{GLPS1}.}.
  \end{theorem}

\begin{proof}
To check that  $H^\dagger$-mod is a sequilinear pivotal category, we need to show the relations in \eqref{eq:Hermitian+dual} hold.

 We first compute the adjoint of
  $\ev_V$.   To do this we fix a basis  $\{e_i\}_i$ of $V$ and let $\{e'_i\}_i$ be the right sesquilinear dual basis
 such that $(e_i\mid e'_j)_V=\delta_i^j$. Then  for all basis elements $e_j$ and $e_k^*$ we have   $$\bp{e_k^*\otimes e_j,\tcoev_V(1)}_{V^*\otimes V}=\bp{e_k^*\otimes e_j,\tau\XX_{V^*, V}\tcoev_V(1)}_p
  =\bp{e_k^*\otimes e_j,\tau\coev_V(1)}_p
 $$
  $$ =\bp{e_k^*\otimes e_j,\sum_i (-1)^{|e_i'|} (e_i')^*\otimes e_i'}_p
  = \sum_i (-1)^{|e_i'|+|e_i'||e_j|} \bp{e_k^*,(e_i')^*}_{V^*} \bp{e_j, e_i'}_{V}$$
  $$= \sum_i (-1)^{|e_i'|+|e_i'||e_j|}\sum_l \wb{e_k^*(e_l)}(e_i')^*(e_l')\delta_{i,j}
  =\sum_{i,l}(-1)^{|e_i|+|e_i||e_j|}\delta_{k,l}\delta_{i,l}\delta_{i,j}=\delta_{k,j}
  =\bp{\ev_V(e_k^*\otimes e_j),1}_\unit,$$
where the second equality follows from \cite[Lemma 4.13]{GLPS1}.  Hence $\tcoev_V=\ev_V^\dagger$.

Next we compute the adjoint of
  $\coev_V$.
  To do this 
  we use that for any vector $v\in V,$ $\wb{\bp{v,e_i'}_V}=e_i^*(v)$.
  Then for all basis elements $e_j$ and $e_k^*$ we have
$$\bp{e_j\otimes e_k^*,\coev_V(1)}_{V\otimes V^*}=\bp{e_j\otimes e_k^*,\tau\XX_{ V,V^*}\coev_V(1)}_p
  =\bp{e_j\otimes e_k^*,\tau\tcoev_V(1)}_p
 $$
 $$
  =\bp{e_j\otimes e_k^*,\sum_i g^{-1}e_i'\otimes (e_i')^*}_p= \sum_i (-1)^{|e_i'||e_k^*|} \bp{e_j, g^{-1}e_i'}_{V}\bp{e_k^*,(e_i')^*}_{V^*}
 $$
 $$
  = \sum_i (-1)^{|e_i'||e_k^*|} \bp{e_j, g^{\dagger}e_i'}_{V}\sum_l \wb{e_k^*(e_l)}(e_i')^*(e_l')
  = \sum_{i,l }(-1)^{|e_i'||e_k^*|} \bp{ge_j, e_i'}_{V}\delta_{k,l}\delta_{i,l}
 $$
  $$
  = (-1)^{|e_k'||e_k^*|} \bp{ge_j, e_k'}_{V}
=(-1)^{|e_k^*||e_j|}\wb{e_k^*(ge_j)}=\bp{\tev_V(e_j\otimes e_k^*),1}_{\unit}
 $$
where in the
second equality we use \cite[Lemma 4.13]{GLPS1}, in the fourth equality we use the fact that the pivotal element $g$ is group-like so is necessarily of degree zero, and in the
second to last equality we use the fact that $|e_k^*|=|e_j|$ if $e_k^*(ge_j)$ is non-zero.
So $\tev_V^\dagger=\coev_V$.

Next, using that the evaluations are determined by the zigzag relations, we have
$$\Id_V=\dag{\bp{(\Id_V\otimes\ev_V)(\coev_V\otimes\Id_V)}}=(\dag\coev_V\otimes\Id_V)(\Id_V\otimes\tcoev_V)$$
so that $\dag\coev_V=\tev_V$ and similarly, $\dag\tcoev_V=\ev_V$.

Finally, we check that $\dag{{c_{V,W}}}={c_{V,W}}^{-1}$.  To do this we will need the fact that
$$
\XX_{V,W} c_{V,W}^{-1}= c_{W,V}^{-1} \XX_{W,V}.
$$
This follows since $\XX$ commutes with $\Delta(\dt)$, $\dt\otimes \dt$, and $\XX$, and because by the definition of $\XX$ we have
 $$
c_{V,W}^{-1}=\dt_{V}\otimes \dt_{W} \XX^{-1}_{V,W} (\dt_{W \otimes V})^{-1}.
$$
Then
\begin{align*}
\label{}
  \bp{v\otimes w,c_{V,W}^{-1}( w'\otimes v')}_{V\otimes W}  & =   \bp{v\otimes w, \tau\XX_{V,W} c_{V,W}^{-1}( w'\otimes v')}_{p}  \\
    &  = \bp{v\otimes w, \tau c_{W,V}^{-1}\XX_{W,V}( w'\otimes v')}_{p}  \\
    &  = \bp{v\otimes w, \tau (R^{-1} \tau) \XX_{W,V}( w'\otimes v')}_{p}  \\
     &  = \bp{v\otimes w, R^{-1}_{21}  \XX_{W,V}( w'\otimes v')}_{p}  \\
      &  = \bp{v\otimes w, R^{\dagger \otimes \dagger}  \XX_{W,V}( w'\otimes v')}_{p}  \\
        &  = \bp{R(v\otimes w),  \XX_{W,V}( w'\otimes v')}_{p}  \\
          &  = \bp{\tau R(v\otimes w), \tau \XX_{W,V}( w'\otimes v')}_{p}  \\
      &=    \bp{c_{V,W}(v\otimes w),w'\otimes v'}_{V\otimes W}
\end{align*}
where $R^{-1}\tau$ and $ R^{-1}_{21}=\tau(R)^{-1}$ are the actions by the corresponding elements.

Finally, we note by assumption that $g^{\dagger}=g^{-1}$.
\end{proof}




We now consider existence and symmetry of compatible sesquilinear forms.
If $V$ is a $\FK$-vector space, let $\wb V=\{\bar v:v\in
V\}$ be the same additive group with antilinear scalar multiplication,
i.e. for any $k\in\FK,v\in V$, one has $\wb k.\wb
v=\wb{k.v}$.  Let $x\mapsto \ldag
x$ be the inverse isomorphism of $\dagger:H\to H$.  Then if
$V$ is an $H$-module, we can define a module $\con
V$ with underlying vector space $\wb V$ and action given by $\forall
x\in H,\forall v\in V$,
\begin{equation} \label{eq:Vbar-action}
x.\wb v= 
\wb{\ldag{S(x)}.v}.
\end{equation}
To see that this is a well-defined action, observe
\begin{align*}
 xy \cdot \wb v
 	&\refequal{\eqref{eq:Vbar-action}}  
    \wb{\ldag{S(xy)}.v}
	 \refequal{\eqref{eq:superanti}} 
             \sig xy\wb{\ldag{(S(y) S(x))}.v}	
	\\ &= 
		\wb{ \ldag{S(x)}\ldag{S(y)} .v}=x\cdot(y\cdot\wb v)
\end{align*}
where we used that $\ldag{(xy)} = \sig xy \ldag{y} \ldag{x}$.
\begin{lemma}\label{L:sesq}
  Let $V$ be an $H$-module. Then there exists a bijection between the
  set of compatible sesquilinear forms $f$ on $V$ and the set of
   even isomorphisms of $H$-modules $\vp:\con V\stackrel{\sim}\to V^*$.
\end{lemma}
\begin{proof}
  The sesquilinear form associated to an isomorphism
  $\vp: \con V\stackrel\sim\to V^*$ is given by
  $$f:V\times V\stackrel{\bar\Id\otimes\Id}\longrightarrow
  \con V\otimes V\stackrel{\vp\otimes\Id}\longrightarrow
  V^*\otimes V\stackrel\ev\longrightarrow \FK ,$$
  where we note that the first $\times$ above is not a tensor product, and $\wb\Id$ is not linear.
  The compatibility follows from:
 \begin{align*}
   f(x.v_1,v_2) &:=\ev(\vp(\wb{x.v_1})\otimes v_2)
                  =\ev(\vp(\wb{\ldag{S}(S^{-1}(\dag{x})).v_1})\otimes v_2)
   \\
                &\refequal{\eqref{eq:Vbar-action}}
                  \ev(\vp(S^{-1}(\dag x).\bar v_1)\otimes v_2)
   \\
                &\refequal{\varphi\text{ intertwines}} 
                  \ev(S^{-1}(\dag x).\vp(\bar v_1)\otimes v_2)
   \\
                & \refequal{\text{def of ev}}
                  [ S^{-1}(\dag x).\vp(\bar v_1)](v_2)
   \\
                &\refequal{\eqref{eq:superdual}}
                  \sig x{v_1}
                  [ \vp(\bar v_1)] (\dag x.v_2)
   \\
                & = 
                  \sig x{v_1}\ev(\vp(\bar v_1)\otimes\dag{x}.v_2)=\sig x{v_1} f(v_1|\dag{x}.v_2),
 \end{align*}
 where we used that $S$ and its inverse commute with $\dagger$ and its inverse, and that $|\varphi(v_1)|=|v_1|$ since $\varphi$ is an even map.
  Conversely,
  given a sesquilinear form $f$, $\vp$ is defined by
  $\vp(\wb v)=f(v|\cdot)$.
\end{proof}

\begin{lemma}\label{L:ProportionalHermitianSimple}
  Assume that $\FK=\C$ and $H$ has a superinvolutive dagger. If $V$
  is an irreducible sesquilinear module, then the form on $V$ is
  proportional to a Hermitian form.
\end{lemma}

\begin{proof}
Let $V_0$ and $H_0$ be the degree zero subspaces of $V$ and $H$ respectively.
  First note that $V_0$ is an irreducible $H_0$-module because any $H_0$-submodule
  $W_0$ of $V_0$ would generate an $H$-submodule of $V$, namely
  $H\cdot V_0$.  Then one can reproduce the proof of \cite[Lemma 4.3]{GLPS1}
    to show that up to rescaling, we can
  assume the restriction of $(\cdot\mid\cdot)$ on $H_0\times H_0$ is
  Hermitian.  Similarly we can assume that the form on $H_1\times H_1$
  is $\lambda$ times a Hermitian form.  Next, since $V_0$ is not an
  $H$-submodule, there exists $v_0\in V_0$, $v_1\in V_1$,
  $h_1\in H_1$ such that
  $$1= (h_1v_0\mid v_1)=(v_0\mid \dag h_1v_1)=\wb{(\dag h_1v_1\mid v_0)}=
  -\wb{( v_1\mid  { h_1^{\dagger\dagger}}v_0)}=-\wb\lambda\lambda^{-1}(h_1^{\dagger\dagger}v_0\mid v_1)=\wb\lambda\lambda^{-1}$$
  so the form is also Hermitian on $V_1$.
\end{proof}
A module $W$ of $H^\dagger$-mod is called {\em Hermitian} if its
sesquilinear form is Hermitian.
The orthogonal direct sum of two sesquilinear modules is a
sesquilinear module.  Reciprocally, we say that a sesquilinear module
is split if it is a direct sum of sesquilinear modules with the induced
sesquilinear forms.  We say that a sesquilinear module is semi-split if 
all its summands can be equipped with compatible sesquilinear
forms (which are not necessarily the restrictions of the original
sesquilinear form).

\begin{proposition}\label{P:TensorProdHerm}
  Let $W_1,W_2$ be Hermitian modules of $H^\dagger$-mod which are
  indecomposable in $H$-mod.  If $W_1\otimes W_2$ is semi-split then
  $W_1\otimes W_2$ is Hermitian.
\end{proposition}
\begin{proof}
  Consider a direct sum decomposition $W_1 \otimes W_2=\bigoplus_iV_i$
  with $V_i$ indecomposable and the dual decomposition
  $V'_i=\bs{v'\in W_1 \otimes W_2 | (v',\bigoplus_{j\neq
      i}V_j)=\bs0}$.

  Fix a factor $V_i$. By Lemma \ref{L:sesq}, since $W_1\otimes W_2$ is semi-split, we have $\con{V_i}\simeq
  V_i^*$.  Similarly, since
  $(\cdot|\cdot)$ restricts to a non-degenerate sesquilinear pairing on $V'_i\times
  V_i$, we have ${V'_i}^*\simeq \con{V_i}$ where the isomorphism is given by
  $$\begin{array}[t]{rcl}
     \con{V_i}&\to&{V'_i}^*\\
     \bar v&\mapsto& \wb{\bp{\cdot,v}}_{\ |V_i}
   \end{array} .$$
    So in particular, $V_i\simeq V'_i$ and the half twist has the same value
    $\brk{\dt_{V_i}}=\brk{\dt_{V_i'}}$, which we denote by $\dt_0\in\C$. Let us call
    $\dt_1=\brk{\dt_{W_1}}$ and
    $\dt_2=\brk{\dt_{W_2}}$ and $\theta_j=\dt_j^2$. Let $(v,v')\in
    V_i\times V'_i$. Then
  \begin{align*}
    (v'|v)&=(v'|\tau\XX(v))_p \\
    	&= (v'|\tau\dt_{W_2 \otimes W_1}^{-1}\tau R(\dt_{W_1}\otimes\dt_{W_2})(v))_p\\
          &=\dt_0^{-1}\dt_1\dt_2\bp{v'|\sqr(\theta_0\Delta^{op}(\theta))
            R(\sqr(\theta_1^{-1}\theta^{-1})
            \otimes\sqr(\theta_2^{-1}\theta^{-1}) (v)) }_p
            \\
          &=\dt_0^{-1}\dt_1\dt_2\bp{\bp{\sqr(\wb\theta_1^{-1}\dagp{\theta^{-1}})
            \otimes\sqr(\wb\theta_2^{-1}\dagp{\theta^{-1}})}R_{21}^{-1}
            \sqr(\wb\theta_0\Delta(\dag{\theta}))
            )v'|v}_p\\
          &=\bp{\dt_0\dt_1^{-1}\dt_2^{-1}\bp{\sqr(\theta_1\theta)
            \otimes\sqr(\theta_2\theta)}R_{21}^{-1}
            \sqr(\theta_0^{-1}\Delta(\theta^{-1}))
            )v'|v}_p\\
          &=\bp{\tau(\dt_{W_2}^{-1}\otimes\dt_{W_1}^{-1})R^{-1}\tau
            \dt_{W_1 \otimes W_2}
            )v'|v}_p=\bp{\tau\XX^{-1}(v')|v}_p\\
          &=\bp{\tau\XX(v')|v}_p=\wb{\bp{v|\tau\XX(v')}_p}=\wb{\bp{v|v'}}.
  \end{align*}
  In the above equalities, the elements $\theta, \theta^{-1}$, etc.,
  are all acting on modules.  Note that the fifth equality above
  follows from $\wb\dt_j=+\dt_j^{-1}$.
\end{proof}

\section{Relative Hermitian-modular and Hermitian-spherical categories} \label{sec:relhermmodsph}
In this section we recall notions related to generically semisimple categories and consider Hermitian structures on these categories.

\subsection{Modified trace in a Hermitian pivotal category}
\subsubsection{Modified traces on projective modules.}  Let $\cat$ be a pivotal $\FK$-category and
let $\Proj$ be the full subcategory of $\cat$
consisting of projective objects.
%
For any objects $V,W$ of $\cat$, and any endomorphism $f$ of $V\otimes
W$, set
\begin{equation}\label{E:trL}
\ptr_{L}(f)=(\ev_{V}\otimes \Id_{W})\circ(\Id_{V^{*}}\otimes
f)\circ(\tcoev_{V}\otimes \Id_{W}) \in \End_{\cat}(W),
\end{equation} and
\begin{equation}\label{E:trR}
\ptr_{R}(f)=(\Id_{V}\otimes \tev_{W}) \circ (f \otimes \Id_{W^{*}})
\circ(\Id_{V}\otimes \coev_{W}) \in \End_{\cat}(V).
\end{equation}

\begin{definition}\label{D:trace}  A \emph{modified trace on $\Proj$} (or \emph{m-trace}) is a family of linear functions
$\{\mt_V:\End_\cat(V)\rightarrow \Bbbk\}$
where $V$ runs over all objects of $\Proj$, such that the following
conditions hold.
\begin{enumerate}
\item  If $U\in \Proj$, and $W\in \ob$, then for any $f\in \End_\cat(U\otimes W)$, we have
\begin{equation}\label{E:VW}
\mt_{U\otimes W}\left(f \right)=\mt_U \left( \ptr_R(f)\right).
\end{equation}
\item  If $U\in \Proj$, and $W\in \ob$, then for any $f\in \End_\cat(W\otimes U)$, we have
\begin{equation}\label{E:VWleft}
\mt_{W\otimes U}\left(f \right)=\mt_U \left( \ptr_L(f)\right).
\end{equation}
\item  If $U,V\in \Proj$, then for any morphisms $f:V\rightarrow U $, and $g:U\rightarrow V$  in $\cat$, we have
\begin{equation}\label{E:fggf}
\mt_V(g\circ f)=\mt_U(f \circ g).
\end{equation}
\end{enumerate}
\end{definition}
If $\mt$ satisfies only \eqref{E:VW} and \eqref{E:fggf}, it is called
a right trace, (and the notion of a left trace is similar). When the category is ribbon, there is no difference
between a right trace and a trace on $\Proj$.  If $\cat$ is unimodular,
there exists up to a scalar a unique right trace on $\Proj$; it is
non-degenerate (cf. Theorem 5.5 of \cite{GKP3}), in the following
way. Let $V,W\in\cat$ with $V$ projective.  Then the pairing
$\brk{\cdot,\cdot}_{V,W}: \Hom_\cat(W,V)\otimes\Hom_\cat(V,W)\to\C$
given by $$\brk{f,g}_{V,W}=\mt_V(fg)$$ is non-degenerate.  It is
symmetric in the following sense. If $W$ is also projective, then
\begin{equation}
  \label{eq:pairing}
  \brk{g,f}_{W,V}=\brk{f,g}_{V,W} \ .
\end{equation}
If $W$ is not projective, then we take \eqref{eq:pairing} as a
definition.

\subsubsection{Renormalized ribbon graph evaluations} \label{subsec:renorm-ribbon-graph}
We now recall the process of renormalizing colored ribbon graphs invariants first used in
\cite{GP1,GP2}.
Let $\cat$ be a $\FK$-linear pivotal category. Following Turaev (\cite{Tu}), by a $\cat$-colored ribbon graph we mean
an oriented surface obtained as the thickening of a graph 
whose oriented edges are colored by objects of $\cat$ and whose
vertices are divided in two sets: the boundary vertices which are univalent and the internal vertices which are
thickened to coupons colored by morphisms of~$\cat$.  Let
$ \Graph_\cat$ be the category of $\cat$-colored ribbon graphs
embedded in $\R \times [0,1]$ and $\Gfun: \Graph_\cat \to \cat$ be the
``planar'' Reshetikhin-Turaev $\FK$-linear
functor 
(see \cite{GPT2}).  When $\cat$ is ribbon, we also denote by $\Gfun$ the original Reshetikhin-Turaev functor from the category of $\cat$-colored ribbon graphs
embedded in $\R^2 \times [0,1]$ to $\cat$.

Let $\cat$ be a $\FK$-linear pivotal (resp. ribbon) category and let
$T\subset S^2$ (resp. $T\subset S^3$) be a closed $\cat$-colored ribbon
graph (i.e. with no boundary vertices).  Let $e$ be an edge of $T$ colored with a projective object $V$ of $\cat$.
Cutting $T$ at a point of $e$, we obtain a $\cat$-colored ribbon graph
$T_V $ in $\R\times [0,1]$ (resp., in $\R^2\times [0,1]$) where
$\Gfun(T_V)\in\End(V)$.  
We call $T_V $ a \emph{cutting
  presentation} of $T$. 
A trace $\mt = \{\mt_{V}\}_{V \in \Proj}$ determines an isotopy
invariant $\Gfun'$ of $\cat$-colored ribbon graphs determined by
\begin{equation} \label{eq:defG'mt}
  \Gfun'(T)=\mt_V(T_V)
\end{equation}
  where $T_V$ is any cutting
presentation of $T$ (see \cite{GKP1}).

Let $\A$ be the set of simple projective objects of $\cat$.  Define  a
function $\qd:\A\to \FK^\times$ by $\qd(V)=\mt_{V}(\Id_V)$ so that $\qd(V)=\qd(V^*)$ for all $V\in \A$.  If $T_V$ is as above with
$V\in\A$, then $\Gfun(T_V)\in\End(V)=\FK \Id_V$. Let
$\ang{T_V} \, \in \FK$ denote the isotopy invariant of $T_V$ defined
from the equality $\Gfun(T_V )= \, \ang{ T_V}\, \Id_V$.  Then
$\Gfun'(T)$ can be computed by:
\begin{equation} \label{eq:defG'}
  \Gfun'(T)=\qd(V)\ang{T_V}.
\end{equation}

\subsubsection{Modified trace and pairing of morphisms}
Let $\cat$ be
a 
Hermitian pivotal category and assume that $\cat$ has a trace on
projective objects.
We say a trace is {\em real} if for any endomorphism $f$ of a projective object, one has
$$\mt(f^\dagger)=\wb{\mt(f)}.$$

\begin{lemma} \label{lem:tracedagger} Assume that there exists a
  simple projective module in $\cal{C}$. Then the trace is proportional to a real
  one. In particular there is a non-zero real trace on $\Proj$ for
  which $\wb{\md(V)}=\md(V)$ for any $V\in\Proj$.
\end{lemma}
\begin{proof}
  Let $V_0$ be a simple projective module.  Then the non-degeneracy of
  the trace implies that $\md(V_0)\neq0$.  Up to rescaling $\mt$, we
  can assume that $\md(V_0)$ is real.  We now show that $\mt$ is
  real. This follows as in \cite[Lemma 4.19]{GLPS1} from the unicity
  of the trace and the fact that
  $f\mapsto \mt^\dagger(f):=\wb{\mt(f^\dagger)}$ is firstly also a
  trace, thus proportional to $\mt$, and secondly
  $\mt^\dagger(\Id_{V_0})=\wb
  {\md(V_0)}=\md(V_0)=\mt(\Id_{V_0})\neq0$. From this we get
  $\mt=\mt^\dagger$.  The last statement follows from
  $\md(V)=\mt_V(\Id_V)$ and $\Id_V^\dagger=\Id_V$.
\end{proof}
From now on, we will always assume if a Hermitian pivotal category is
equipped with a trace, then it is real.
Then one has the following.
\begin{proposition}
  For any projective object $V$ and any object $W$ of $\cat$, the following are true.
  \begin{enumerate}
  \item There exists a non-degenerate bilinear pairing
    $$\Hom(W,V)\otimes_\FK\Hom(V,W)\to\FK$$
    $$f\otimes g\mapsto \mt_V( fg).$$
  \item There exists a non-degenerate Hermitian form on $\Hom(V,W)$ given by
    $$(f\mid g)=\mt_V(\dag f g).$$
  \end{enumerate}
\end{proposition}

\subsection{Generically semisimple categories} \label{sec:GenHerm}
Here we recall the definition of generically semisimple categories
that lead to Turaev-Viro type and Witten-Reshethikin-Turaev type non-semisimple topological invariants.

Let $\FK$ be a field and let $\Gr$ be an abelian group.
A pivotal $\FK$-category is {\em $\Gr$-graded} if
  for each $g\in \Gr$ we have a non-empty full subcategory $\cat_g$ of
  $\cat$ such that the following hold.
  \begin{enumerate}
  \item $\unit \in \cat_e$, (where $e$ is the identity element of $\Gr$).
  \item  $\cat=\bigoplus_{g\in\Gr}\cat_g$.
  \item  if $V\in\cat_g$, $V'\in\cat_{g'}$ then $V\otimes
    V'\in\cat_{g+g'}$.
  \item  if $V\in\cat_g$, $V'\in\cat_{g'}$ and $\Hom_\cat(V,V')\neq 0$, then
    $g=g'$.
  \end{enumerate}
For a subset $\X\subset\Gr$ we say:
\begin{enumerate}
\item $\X$ is \emph{symmetric} if $-\X^{}=\X$.
\item $\X$ is \emph{small} in $\Gr$ if the group $\Gr$ can not be covered by a
  finite number of translated copies of $\X$, in other words, for any
  $ g_1,\ldots ,g_n\in \Gr$, we have $ \bigcup_{i=1}^n (g_i+\X) \neq\Gr$.
\end{enumerate}

\begin{enumerate}
\item A $\FK$-category $\cat$ is \emph{semisimple} if all its objects are semisimple.
\item A $\FK$-category $\cat$ is \emph{finitely semisimple} if it is
  semisimple and has finitely many isomorphism classes of simple
  objects.
\item A $\Gr$-graded category $\cat$ is a \emph{generically
    $\Gr$-semisimple} category (resp.  \emph{generically finitely
    $\Gr$-semisimple} category) 
  if there exists a small symmetric subset $\X\subset \Gr$ such that
  for each $g\in\Gr\setminus\X$, $\cat_g$ is semisimple (resp.
  finitely semisimple). 
  By a \emph{generic simple} object we mean a simple object of
  $\cat_g$ for some $g\in\Gr\setminus\X$.
\end{enumerate}




\subsubsection{Relative Hermitian-spherical categories}

We now introduce the notion of an \emph{$(\X,\qd)$-relative
  $\Gr$-spherical} category which is equivalent to the one given in \cite{GPT2}.
\begin{definition}\label{D:G-spherical}
 Let $(\cat,\Gr,\X)$ be a generically finitely
    $\Gr$-semisimple category $\FK$-linear pivotal category.
 Let $\A$ be the class of all simple
  generic objects. We say that $\cat$ is
  \emph{$(\X,\qd)$-relative $\Gr$-spherical} if  the following hold. 
  \begin{enumerate}
  \item \label{ID:G-sph6} There exists a trace
    $\{\mt_V:\End_\cat(V)\rightarrow \Bbbk\}$ where $V$ runs over all
    objects of $\Proj$ so that the function $\qd:\A\to \FK^*$ is
    given by $\qd(V)=\mt_{V}(\Id_V)$.

  \item \label{ID:G-sph7} There exists a map $\bb:\A\to \FK^*$ such that
    $\bb(V)=\bb(V^*)$, $\bb(V)=\bb(V')$ for any isomorphic objects $V, V'\in
    \A$ and for any $g_1,g_2,g_1+g_2\in \Gr\setminus\X$ and $V\in \cat_{g_1+g_2}$
    we have
  \begin{equation}\label{eq:bb}
    \bb(V)=\sum_{V_1\in irr(\cat_{g_1}),\, V_2\in irr(\cat_{g_2})}
    \bb({V_1})\bb({V_2})\dim_\FK(\Hom_\cat(V, V_1\otimes V_2))
  \end{equation}
  where $irr(\cat_{g_i})$ denotes a representing set of the
  isomorphism classes of simple objects of $\cat_{g_i}$.
  \end{enumerate}
  If $\cat$ is a category with the above data, for brevity we say
  $\cat$ is a \emph{relative $\Gr$-spherical category}.
\end{definition}
The map $\bb$ always exists when $\FK$ is a field of characteristic $0$ and
$\cat$ is a category whose objects are finite-dimensional $\FK$-vector spaces.
In particular, in \cite{GPT2} it is shown that, for any $g\in \Gr\setminus\X$, the map
\begin{equation}\label{E:defbb}
   \bb(V)=\dim_\FK(V)/\left(\sum_{V'\in irr(\cat_g)}\dim_\FK(V')^2\right)
\end{equation}
is well-defined and satisfies all the properties above.  In particular, if for any generic degree $g$, $\cat_g$ has $N$ simple objects of dimension $d$, then $\bb$ can be chosen to be the constant map $\bb(V)=\frac1{Nd}$.

\begin{definition}
A relative $\Gr$-spherical category is called a \emph{relative
    $\Gr$-Hermitian-spherical category}, if it is simultaneously a Hermitian pivotal category, a relative
  $\Gr$-spherical category and the trace $\mt$ is real.
\end{definition}
\subsubsection{Relative Hermitian-premodular categories}

In \cite{D17} De Renzi gives the notion of a relative modular category
and shows that such categories lead to TQFTs.  A $\Gr$-relative
modular category is a ribbon category which is a generically
$\Gr$-semisimple category $\cat$ with symmetric small set $\X$, has a free
realization and non-zero m-trace satisfying compatibility conditions.
Let us recall this definition.

Let $\Zt$ be an abelian group.  A \textit{free realization of $\Zt$ in
  a ribbon category $\cat$} is a monoidal functor
$\sigma : \Zt \rightarrow \cat$, where $\Zt$ also denotes the discrete
category over $\Zt$ with tensor product given by the group operation
$+$, satisfying $\theta_{\sigma(k)} = \id_{\sigma(k)}$ for every
$k \in \Zt$, and inducing a free action on isomorphism classes of
simple objects of $\cat$ by tensor product with $\sigma(k)$.  Recall
the definition of a generically semisimple category at the beginning
of Subsection \ref{sec:GenHerm}.

\begin{definition}[\cite{D17}] \label{def:X} If $\Gr$ and $\Zt$ are
  abelian groups, and if $\X \subset \Gr$ is a small symmetric subset,
  then a \textit{$(\Zt,\X)$-relative $\Gr$-premodular category} is a
  generically semisimple category $(\cat,\Gr,\X)$ which is ribbon,
  $\Bbbk$-linear, has a free realization
  $\sigma : \Zt \rightarrow \cat_0$, and a non-zero m-trace $\mt$ on
  the projective objects of $\cat$. This data is subject to the
  following conditions.
  \begin{enumerate}
  \item \textit{Finitely many orbits}. For every
    $g \in \Gr \smallsetminus \X$, the homogeneous subcategory $\cat_g$
    is semisimple and dominated by
    $\Theta(\cat_g) \otimes \sigma(\PGr)$ for some finite set
    $\Theta(\cat_g) = \{ V_i \in \cat_g \mid i \in I_g \}$ of simple
    projective objects with epic evaluation.  Here, dominated means
    that the identity morphism of every object in the category can be
    written as a linear combination of maps which factor through
    objects in the dominating set.
  \item \label{X:2}\textit{Compatibility}. There exists a bilinear map
    $\psi : \Gr \times \Zt \rightarrow \Bbbk^*$ such that
    \begin{equation} \label{eq:psi}
    c_{\sigma(k),V} \circ c_{V,\sigma(k)} = \psi(g,k) \cdot \id_{V
      \otimes \sigma(k)}
      \end{equation}
      for every $g \in \Gr$, for every
    $V \in \cat_g$, and for every $k \in \Zt$.
 \end{enumerate}
\end{definition}

If $\cat$ is a category with the above data, for brevity we say $\cat$
is a relative $\Gr$-premodular category. If $\cat$ is a
$(\Zt,\X)$-premodular $\Gr$-category then the associated \textit{Kirby
  color of index $g \in \Gr \smallsetminus \X$} is the formal linear
combination of objects
\[
 \Omega_g := \sum_{i \in I_g} \md(V_i) \cdot V_i.
\]
In particular, there exist constants
$\Delta_{-\Omega},\Delta_{+\Omega} \in \Bbbk$, called
\textit{stabilization coefficients}, which realize the skein
equivalences of Figure \ref{F:stabilization_coefficients_Omega}, and
which are independent of both $V \in \cat_g$ and
$g \in \Gr \smallsetminus \X$. We say the relative premodular
category $\cat$ is \textit{non-\-de\-gen\-er\-ate} if
$\Delta_{-\Omega} \Delta_{+\Omega} \neq 0$.
\begin{figure}[tb]
 \centering
 \includegraphics{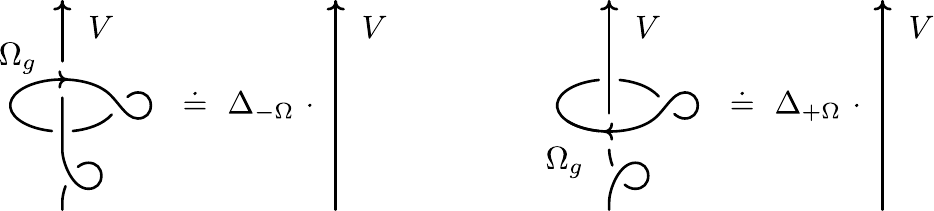}
 \caption{Skein equivalences defining $\Delta_{-\Omega}$ and $\Delta_{+\Omega}$.}
 \label{F:stabilization_coefficients_Omega}
\end{figure}

A non-degenerate relative $\Gr$-premodular category
gives rise to the construction of a
3-manifold invariant, (see for example \cite[Section 4.3]{CGP1}, where it is called relative $\Gr$-modular rather than the modern term relative $\Gr$-premodular).
A relative $\Gr$-modular category is a relative $\Gr$-premodular
category which satisfies the modularity condition given in \cite{D17}.
Note that relative $\Gr$-modular categories are automatically non-degenerate.
Such a category gives rise to a TQFT.  The exact form of the
modularity condition is not critical for this paper.  Next we give a
property which implies the TQFT is Hermitian.

\begin{definition}\label{D:GenericHermitian}
  A {\em relative Hermitian-(pre)modular category} is a Hermitian ribbon category which is also a relative (pre)modular category, such that the trace on projectives is real.
\end{definition}

\newcommand{\Herm}{{\operatorname{Herm}}}
\subsection{Relative Hermitian-premodular category in a sequilinear one}\label{s:sesqtoHerm}
Let $\cat$ be a sesquilinear pivotal category. A collection of objects
is called a {\em collection of Hermitian objects} in $\cat$ if the following hold (in the following, we call an element of the collection a {\em Hermitian object}).
\begin{enumerate}
\item For any two Hermitian objects $V,W$ and any
  $f\in\Hom_\cat(V,W)$, we have $(f^{\dagger})^{\dagger} = f$.
\item For any simple $V\in\cat$, $V$ is isomorphic to a Hermitian object.
\item If the tensor product of two Hermitian simple objects is
  semisimple, then it is Hermitian.
\end{enumerate}

If $S$ is a class of objects in a pivotal category $\cat$, then we
define the \emph{category generated} by $S$ as the full subcategory of
$\cat$ which has as objects, 
all tensor products of the form:
$$X_{1}\otimes X_{2}\otimes \cdots \otimes X_{p}\quad\text{ where }p\in\N\text{ and }X_i\in S\cup S^*.$$




\begin{theorem}
  If $\cat$ is a generically semisimple, sesquilinear pivotal category
  which has a collection of Hermitian objects where $\ev_{V}$ is an
  epimorphism for each generic simple Hermitian object $V$, then the
  simple Hermitian objects generate a
  subcategory $\cat^\Herm$ of $\cat$ which is a generically semisimple,
  Hermitian pivotal category.
  \end{theorem}
\begin{proof}
  The generically semisimple sesquilinear pivotal structure of $\cat$
  provides the same structure on $\cat^\Herm$.  Thus to prove the
  theorem we need to show that $(f^\dagger)^\dagger=f$ for any
  morphism $f$ in $\cat^\Herm$.  Let
  $f:W_1\otimes\cdots\otimes W_m\to W'_1\otimes\cdots\otimes W'_n$ be
  a morphism in $\cat^\Herm$ where $W_1,\ldots, W_m,W'_1,\ldots, W'_n$
  are Hermitian simple objects.

  We first consider the following ``regular'' case: Assume
  $m\ge1,n\ge1$ and for any $i\in\{1,\ldots m\}$, any
  $j\in\{1,\ldots n\}$, $W_1\otimes\cdots\otimes W_i$ and
  $W'_1\otimes\cdots\otimes W'_j$ are in generic degrees.  Then we
  prove $(f^\dagger)^\dagger=f$ by induction on $k=m+n$.

  The base case is $k=2$, which means that $m=n=1$. Then by the first
  axiom of a collection of Hermitian objects, $(f^\dagger)^\dagger=f$ holds.

  Assume true for $k\geq 2$ and consider a morphism $f$ where
  $m+n=k+1$. \\
  If $m\ge2$ then since $W_1\otimes W_2$ is generic, it is isomorphic
  to a direct sum of simple Hermitian modules $U_i$. Note that
  $(U_i,W_3,\ldots,W_m)$ still satisfies the ``regular'' case
  condition.  Let $g_i : W_1\otimes W_2 \to U_i$ be the projection
  morphisms corresponding to this direct sum which satisfy
  $(g_i^{\dagger})^{\dagger}=g_i$.  Then there exist morphisms
  $h_i: U_i \otimes W_3\otimes\cdots\otimes W_m\to
  W'_1\otimes\cdots\otimes W'_n$ such that
  $f=\sum_{i}h_i(g_i\otimes \Id_{W_3\otimes\cdots\otimes W_m}).$ Now
  by induction, $h_i$ satisfies $(h_i^\dagger)^\dagger=h_i$ and thus
  $(f^{\dagger})^{\dagger}=f$.\\
  If $n\ge2$, then we can proceed similarly with the $W'_i$. Since
  $W'_1\otimes W'_2$ is generic, it is isomorphic to a direct sum of
  simple Hermitian modules $U_i$.  Let $g_i : U_i\to W'_1\otimes W'_2$
  be the injection morphisms corresponding to this direct sum which
  satisfy $(g_i^{\dagger})^{\dagger}=g_i$.  Then there exist morphisms
  $h_i: W_1\otimes\cdots\otimes W_m\to U_i \otimes
  W'_3\otimes\cdots\otimes W'_n$ such that
  $f=\sum_{i}(g_i\otimes \Id_{W'_3\otimes\cdots\otimes W'_m})h_i.$
  Again by induction, $(h_i^{\dagger})^{\dagger}=h_i$ thus $(f^{\dagger})^{\dagger}=f$.

  Hence the theorem is proven in the ``regular'' case.  Now for the
  general case, let $V$ be a Hermitian simple generic object such that
  for any $i\in\{1,\ldots m\}$, and any $j\in\{1,\ldots n\}$,
  $V\otimes W_1\otimes\cdots\otimes W_i$ and
  $V\otimes W'_1\otimes\cdots\otimes W'_j$ are in generic degrees
  (such a $V$ exists because $\cat$ is generically semisimple with
  respect to a small set $\X$, and by assumption, any simple is
  isomorphic to a Hermitian object).  Then $\Id_V\otimes f$ is in the
  ``regular'' case so that
  $((\Id_V\otimes f)^\dagger)^\dagger=\Id_V\otimes f $.

  But $((\Id_V\otimes f)^\dagger)^\dagger=\Id_V\otimes (f^\dagger)^\dagger$ and since $\ev_{V}$ is an epimorphism, we have
\begin{equation}\label{E:IdfDagger2b}
\Id_V\otimes (f^\dagger)^\dagger=\Id_V\otimes f \text{ if and only if } (f^{\dagger})^{\dagger} = f.
\end{equation}
This complete the proof.
\end{proof}

\begin{corollary}\label{c:hpremod}
  If a relative premodular category $\cat$ is a sesquilinear ribbon
  category with a collection of Hermitian objects containing
  $\sigma(\Zt)$, then $\cat^\Herm$ is a Hermitian relative premodular
  category.
\end{corollary}

\newcommand{\U}{{\mathcal U}}
\newcommand{\Xs}{{\mathcal X}}
\newcommand{\Unitary}{{\mathbb U}}
\subsection{Relative Hermitian-spherical category associated to a relative premodular category}
We describe some machinery on how to obtain a relative Hermitian-spherical category from a relative Hermitian-premodular category.
Assume that $\cat$ 
is a relative Hermitian-premodular category with grading group
$(\Gr,+)$, singular set $\Xs$, and translation group $(\Zt,+)$.  Let
$\Unitary=\{z\in\FK:\bar z z=1\}$.  Note that applying the dagger to
\eqref{eq:psi}, we get that the bilinear map $\psi$ from Definition~\ref{def:X} \eqref{X:2} takes values in $\Unitary$.

We assume the following additional hypothesis:
\begin{equation}
  \label{eq:sqrpsi}
  \exists \sqrt\psi:\Gr\times\Zt\to\Unitary,\text{ which is bilinear and } \sqrt\psi^2=\psi,
\end{equation}
\begin{equation}
  \label{eq:trivialbraiding}
  \forall \lambda,\mu\in\Zt, \quad c_{\sigma_\lambda,\sigma_\mu}=\Id_{\sigma_{\lambda+\mu}} .
\end{equation}

\begin{remark}\
  \begin{enumerate}
  \item If $\Zt/2\Zt$ is finite, then one can replace $\Zt$ with
    $2\Zt$.  Then the category is still relative premodular and one
    can define $\sqrt\psi(g,2\lambda):=\psi(g,\lambda)$.
  \item Identity \eqref{eq:trivialbraiding} with
    $\theta_{\sigma_\lambda}=\Id$ implies that
    $\dim_\cat(\sigma_\lambda)=1$.
     \item A third consequence of \eqref{eq:trivialbraiding} and
    $\sigma_\lambda\otimes\sigma_{-\lambda}=\unit$ is the existence of
    canonical isomorphisms
    $$(\Id_{\sigma_{-\lambda}}\otimes\ev_{\sigma_\lambda})
    (\Id_{\sigma_{-\lambda}\otimes\sigma_\lambda}\otimes\Id_{\sigma_{\lambda}^*})=
    (\tev_{\sigma_\lambda}\otimes\Id_{\sigma_{-\lambda}})
    (\Id_{\sigma_\lambda^*}\otimes\Id_{\sigma_{\lambda}\otimes{\sigma_{-\lambda}}})
    :\sigma_\lambda^*\stackrel{\cong}{\longrightarrow}\sigma_{-\lambda}.$$
    Through these isomorphisms, all (co)evaluation morphisms of the
    objects $\sigma_\lambda$ correspond to the identity of $\unit$.
    We use these isomorphisms to freely identify $\sigma_{\lambda}^*$ and
    $\sigma_{-\lambda}$.
  \end{enumerate}
\end{remark}

\begin{definition}
  For any $\lambda\in\Zt,$ and any $V\in\cat_g$, we let
  $$\XX_{\sigma_\lambda, V}=\sqrt\psi(g,\lambda)^{-1}c_{\sigma_\lambda, V}
  \et\XX_{
    V,\sigma_\lambda}=\sqrt\psi(g,\lambda)^{-1}c_{V,\sigma_\lambda}.$$
  Then $\XX_{\bullet,\sigma_\lambda}$ is a half braiding and
  $\XX_{\bullet,\sigma_\lambda}^\dagger=\XX_{\sigma_\lambda,
    \bullet}=\XX_{\bullet,\sigma_\lambda}^{-1}$.
\end{definition}


\begin{definition}\label{d:catZ}
  Let $\cat^\Zt$ be the category with the same objects as $\cat$ and
  whose morphisms are the $\Zt$-graded vector spaces given by
  $$\Hom_{\cat^\Zt}(V,W)=\bigoplus_{\lambda\in\Zt}\Hom_{\cat}(V,\sigma_\lambda\otimes W).$$
  For a morphism $f:V\to W$, we denote by
  $f_\lambda\in\Hom_{\cat}(V,\sigma_\lambda\otimes W)$ its degree
  $\lambda$ component. The composition of morphisms $f:V_1\to V_2$
  with $g:V_2\to V_3$ in degree $\lambda$ is the essentially finite
  sum
\begin{equation} \label{eq:CZ-comp}
    (g\circ f)_\lambda=\sum_{\lambda'+\lambda''=\lambda}
    (\Id_{\sigma_{\lambda'}}\otimes g_{\lambda''})f_{\lambda'}.
\end{equation}
  The tensor product of $f:V_1\to V_2$ and $h:V_3\to V_4$  is given by
\begin{equation} \label{eq:CZ-tensor}
  (f\otimes h)_\lambda=\sum_{\lambda'+\lambda''=\lambda}
  \bp{\Id_{\sigma_{\lambda'}}\otimes\XX_{V_2,\sigma_{\lambda''}}\otimes\Id_{V_4}}
  \bp{f_{{\lambda'}}\otimes h_{\lambda''}}.
\end{equation}
\end{definition}

\begin{proposition}
  There is a unique structure of a pivotal category on $\cat^\Zt$ such
  that the embedding $F^\Zt:\cat\to\cat^\Zt$ given by considering all
  morphisms of $\cat$ as degree zero morphisms of $\cat^\Zt$ is a
  pivotal functor.  The category $\cat^\Zt$ has a natural Hermitian
  structure compatible with $F^\Zt$.
\end{proposition}
\begin{proof}
  The associativity of composition comes from the strict relation
  $\sigma_{\lambda_1}\otimes
  \sigma_{\lambda_2}=\sigma_{\lambda_1+\lambda_2}$.  The associativity
  and the functoriality of the tensor product come from the fact that
  $\XX_{\bullet,\sigma_{\lambda}}$ are half braidings and Equation \eqref{eq:trivialbraiding} ensures that $f\otimes g=(f\otimes\Id)(\Id\otimes g)$.  The functor
  $F^\Zt$ is clearly monoidal since $\XX_{\bullet,\unit}$ is the
  identity functor.  The category $\cat^\Zt$ inherits the evaluation and coevaluation morphisms of Definition \ref{pivdef} from $\cat$.

  To see that $\cat^\Zt$ is pivotal, we need to
  check that the left dual $f^*$  and right dual $^*f$ of a $\cat^\Zt$-morphism
  $f:V\to W$ are equal.  Let us recall the definitions of   $f^*$ and $^*f$:
  $$f^*=(\ev_{W}\otimes \Id_{V^*})(\Id_{W^*}\otimes f \otimes \Id_{V^*})(\Id_{W^*}\otimes \coev_{V}),$$
and
$$ ^*f=(\Id_{V^*}\otimes \tev_W)(\Id_{V^*}\otimes f \otimes \Id_{W^*})(\tcoev_{V} \otimes \Id_{W^*}).$$

To compute these morphism in $\cat^\Zt$ we need to use the definitions of the composition and tensor products for $\cat^\Zt$ given in \eqref{eq:CZ-comp} and \eqref{eq:CZ-tensor}.  For example,
$$
(\ev_{W}^{\cat^\Zt} (\Id_{W^*}\otimes f))_\lambda
=(\Id_{\sigma_{\lambda}}\otimes \ev_{W})((\XX_{W^*,\sigma_{\lambda}}\otimes \Id_{W})(\Id_{W^*}\otimes f_\lambda))
$$
  where $\ev_{W}^{\cat^\Zt} $ denotes the evaluation in $\cat^\Zt$ and the last composition of maps is in $\cat$.  It follows that as morphisms of $\cat$ we have
  \begin{align*}
  (f^*)_\lambda
 &=\big(\big[(\Id_{\sigma_{\lambda}}\otimes \ev_{W})(\XX_{W^*,\sigma_{\lambda}}\otimes \Id_{W})(\Id_{W^*}\otimes f_\lambda)\big]\otimes \Id_{V^*}\big)(\Id_{W^*}\otimes \coev_{V})\\
 & =(\Id_{\sigma_{\lambda}}\otimes (f_\lambda)^*)(\XX_{W^*,\sigma_{\lambda}}\otimes \Id_{\sigma_{\lambda}^*})(\Id_{W^*}\otimes  \coev_{\sigma_{\lambda}})\\
  & =(\Id_{\sigma_{\lambda}}\otimes\, ^*(f_\lambda))(\XX_{W^*,\sigma_{\lambda}}\otimes \Id_{\sigma_{\lambda}^*})(\Id_{W^*}\otimes  \coev_{\sigma_{\lambda}})\\
  & = (\XX_{V^*,\sigma_{\lambda}}\otimes \tev_{W})(\Id_{V^*}\otimes f_\lambda \otimes \Id_{W^*})(\tcoev_V\otimes \Id_{W^*}) =(^*f)_\lambda
\end{align*}
  where the last equality follows from the naturality of
  $\XX_{\bullet,\sigma_\lambda}$ and $\theta_{\sigma_\lambda}=\Id$.

In order to construct the Hermitian structure on $\cat^\Zt$, let $ f \in \Hom_{\cat^\Zt}(V_1,V_2)$.
Now define a morphism $ f^{\dagger} $ in $\cat^\Zt$ by
\begin{equation}
(f^{\dagger})_{\lambda} := (\Id_{\sigma_{\lambda}} \otimes (f_{-\lambda})^{\dagger} ) \circ  (\tcoev_{\sigma_{-\lambda}} \otimes \Id_{V_2}) \colon
V_2 \rightarrow \sigma_{\lambda} \otimes V_1.
\end{equation}
In graphical terms,
\begin{equation} \label{defofCzfdagger}
(f^{\dagger})_{\lambda} :=
\hackcenter{ \begin{tikzpicture} [scale=.7,  ]
\draw[very thick, postaction={decorate}] (0,1.25) -- (0,-.25);
\draw[very thick, postaction={decorate}] (1.5,1.25) -- (1.5,.25);
\draw[very thick, ] (1.5,0) .. controls ++(0,.25) and ++(0,.5) ..  (2,-1.25);
\draw[very thick] (0,-.25) .. controls ++(0,-1.25) and ++(-.1,-1.25) .. (1.1,-.25);
\node at (0,1.5) {$\scs \sigma_{\lambda} $};
\node at (1.5,1.5) {$\scs V_1 $};
\node at (2,-1.5) {$\scs V_2$};
\node[draw, fill=black!20 ,rounded corners ] at (1.5,0) {$ (f_{-\lambda})^{\dagger}$};
\end{tikzpicture}}
\maps V_2 \rightarrow \sigma_{\lambda} \otimes V_1.
\end{equation}
Using this graphical calculus, it is straightforward to check that the dagger is now a contravariant functor on $\cat^\Zt$.

In order to check that the functor is monoidal, let $ f \colon V_1 \rightarrow V_2$ and $h \colon V_3 \rightarrow V_4$.
Unraveling the definitions, we get
\begin{equation}
((f \otimes h)^{\dagger})_{\lambda} =
\sum_{\lambda'+\lambda''=-\lambda}  (\Id_{\sigma_{-\lambda''-\lambda'}}  \otimes (((f_{\lambda'})^{\dagger} \otimes (h_{\lambda''})^{\dagger}  ) ) \circ (\Id_{\sigma_{\lambda'}} \otimes \XX^{\dagger}_{V_2, \sigma_{\lambda''}} \otimes \Id_{V_4} )) \circ (\tcoev_{\sigma_{\lambda'} \otimes \sigma_{\lambda''}} \otimes \Id_{V_2 \otimes V_4} ).
\end{equation}
On the other hand,
\begin{equation}
(f^{\dagger} \otimes h^{\dagger})_{\lambda} =
\sum_{\lambda'+\lambda''=\lambda}
(\Id_{\sigma_{\lambda'}} \otimes \XX_{V_1,\sigma_{\lambda''}} \otimes (h_{-\lambda''})^{\dagger}) \circ
(\Id_{\sigma_{\lambda'}} \otimes (f_{-\lambda'})^{\dagger}  \otimes \tcoev_{\sigma_{-\lambda''}} \otimes \Id_{V_4}) \circ
(\tcoev_{\sigma_{-\lambda'}} \otimes \Id_{V_2 \otimes V_4} ).
\end{equation}
One could most easily show $((f \otimes h)^{\dagger})_{\lambda} = (f^{\dagger} \otimes h^{\dagger})_{\lambda} $ in a graphical way by manipulating these expressions
using the fact $\XX^{\dagger}=\XX$, using the naturality of $\XX_{\bullet,\sigma_\lambda}$, and reindexing the sum.

The conditions of Definition \ref{sespivdef} are satisfied in $\cat^\Zt$ since they are satisfied in $\cat$ by assumption.
\end{proof}
The embedding $F^\Zt$ is not a full subcategory since
$\cat^\Zt$ has more morphisms.  As a consequence, the braiding of
$\cat$ is no longer dinatural in $\cat^\Zt$ but we will see that
$\cat^\Zt$ is generically finitely semisimple.

\begin{theorem}\label{T:relativeHer-premodularIsSpherical}
 If  $\cat$ is a relative Hermitian-premodular category with  $\sqrt\psi $ as in \eqref{eq:sqrpsi}, and the braiding assumption as in \eqref{eq:trivialbraiding}, then if
  $\cat^\Zt$ has a map $\bb$ as in Equation
  \eqref{eq:bb} it is a relative Hermitian-spherical category.
\end{theorem}
\begin{proof}
  We note that a $\Zt$-orbit of a simple object in $\cat$ becomes
  an isomorphism class of simple objects in $\cat^\Zt$.  Hence
  $\cat^\Zt=\bigoplus_{g\in\Gr}\cat^\Zt_g$ is finitely semisimple in
  each degree $g\in\Gr\setminus\Xs$.
\end{proof}

\section{Relative LW-data from relative Hermitian-spherical categories} \label{sec:relHermtodata}

In this section we show how the notion of a modified trace on a generically Hermitian category can be used to define the basic data needed for the non-semisimple Levin-Wen model from Section~\ref{sec:relativeLW}.

\subsection{Basic data and basis} \label{sec:basis}
Let $\cat$ be a relative $\Gr$-Hermitian-spherical category
 with a family of representative isomorphism classes of
generic simples objects $(V_i)_{i\in I}$. For $g\in \Gr\setminus\X$, let $I_g$ be the finite
set
\begin{equation} \label{eq:Ig}
  I_g=\bs{i\in I:V_i\in\cat_g}
\end{equation}
and call $g$ the degree of $i\in I_g$.
For $(i,j,k)\in I^3$, define the multiplicity
space $$\H^{ijk}=\Hom_\cat(\unit,V_i\otimes V_j\otimes V_k).$$ We call
this dimension the branching level:
\begin{equation} \label{eq:def-delta}
\delta_{ijk}=\dim_\FK(\H^{ijk}).
\end{equation}
Note that $\delta_{ijk}$ is zero unless the sum of the degrees
of $V_i$, $V_j$ and $V_k$ is zero in $\Gr$.  There is an involution on
$I$, $i\mapsto i^*$ uniquely determined by $V_{i^*}\simeq V_i^*$.

{\em Basic data} in $\cat$ is a family $(V_i)_{i\in I}$ as above
equipped with isomorphisms $\bp{w_i:V_i\to V_{i^*}^*}_{i\in I}$
satisfying $w_{i^*}=\bp{w_i}^*\circ\phi_{V_{i^*}}$, or equivalently
\begin{equation}
  \label{eq-w_i}
  \ev_{V_i}(w_{i^*}\otimes\Id_{V_i})=\tev_{V_{i^*}}(\Id_{V_{i^*}}\otimes w_{i})
  \maps V_{i^*}\otimes V_{i}\to\unit .
\end{equation}
The basic data induces non-degenerate bilinear pairings
$$\Theta=\bp{
  \begin{array}{rcl}
    \Theta_{ijk}:\H^{k^*j^*i^*}\otimes \H^{ijk}&\to&\FK\\
    y'\otimes y&\mapsto&\mt_{V_i\otimes V_j\otimes V_k}
                         \left(y(y')^*(w_i\otimes w_j\otimes w_k)\right)
  \end{array}}_{ijk\in I^3}$$
which satisfy $\Theta_{ijk}(y',y)=\Theta_{k^*j^*i^*}(y,y')$.
This pairing is also invariant by pivotal cyclic permutation as we now see. Define
$\rot_{ijk}:\H^{ijk}\to\H^{jki}$ by
$$\rot_{ijk}(y)=\bp{\ev_{V_i}\otimes \Id_{V_j\otimes V_k\otimes V_i}}
  \bp{\Id_{V_i^*}\otimes y\otimes \Id_{V_i}}\bp{\tcoev_{V_i}} .$$
Let $y \in \H^{ijk}$ and $y' \in \H^{i^* k^* j^*}$.
Then $\Theta_{jki}(y',\rot_{ijk}(y))=\Theta_{ijk}(\rot_{i^*k^*j^*}(y'),y)$.

{\em Hermitian basic data} is basic data together with a basis
$$(\y_n^{ijk})_{n=1,\ldots, \delta_{ijk}}\text{ of
}\H^{ijk}$$ such that
\begin{equation}
  \label{eq:dualbasis}
  \Theta_{ijk}\bp{\y_m^{k^*j^*i^*},\y_n^{ijk}}=\delta_m^n\ ,
\end{equation}
\begin{equation}
  \label{eq:unit-basis}
  (\y_n^{ijk}\mid \y_m^{ijk})=\left|
    \begin{array}{l}
      0\text{ if }m\neq n\\
      \gamma_n^{ijk}\in\FK^*\text{ if }m=n
    \end{array}\right.,
\end{equation}
\begin{equation}
  \label{eq:rot-basis}
  \rot_{ijk}(\y_n^{ijk})=\y_n^{kij}.
\end{equation}

We call $i,j,k$ the {\em string labels} and $n$ the {\em branching label} of
$\y_n^{ijk}$. By convention, for $n> \delta_{ijk}$ we set
$\y_n^{ijk}=0\in\H^{ijk}$.

Equation \eqref{eq:unit-basis} determines $(\wb\cdot)$-invariant scalars
$\gamma_n^{ijk}\in\FK^*$.  Another family of such scalars
$\beta(i)\in\FK^*$ is defined by
\begin{equation} \label{def:beta}
\dag w_i=\beta(i) w_i^{-1}\text{ so that }(w_i\mid w_i)=\beta(i) \md(V_i) .
\end{equation}
\begin{proposition} \label{prop:Herm-basicdata} If $\X$ contains all
  the order 2 and order 3 elements of $\Gr$, then the relative
  $\Gr$-Hermitian-spherical category $\cat$ has Hermitian basic data.
\end{proposition}
\begin{proof}
  Let $\A$ be the set of generic labels and choose a total order on
  $\A$.  The assumption on order $2$ elements implies for $i\in \A$,
  $i^*\neq i$. For each $i\in \A$ with $i<i^*$, choose an isomorphism
  $w_i:V_i\to V_{i^*}^*$ and define
  $$w_{i^*}=(\Id_{V_i^*}\otimes\ev_{V_{i}})(\Id_{V_{i}^*}\otimes w_i\otimes
  \Id_{V_{i^*}})(\tcoev_{V_i}\otimes\Id_{V_{i^*}}) :V_{i^*}\to V_{i}^*.$$
  This ensures that \eqref{eq-w_i} is satisfied for $i$ and $i^*$.

  Choose now a total order on
  $\A_3=\bs{(i,j,k)\in \A^3:\delta_{ijk}>0}$. The assumption on order
  $3$ elements implies that for any $i\in \A$, $(i,i,i)\notin \A_3$.
  The dihedral group $D_{2\times3}$ acts on $\A_3$ by
  $(i,j,k)\mapsto(j,k,i)$ and $(i,j,k)\mapsto(k^*,j^*,i^*)$.  This
  action is free so there are $6$ triples in each orbit.  For each
  orbit, pick $(i,j,k)$ to be the smallest element of this orbit and
  choose an orthogonal basis $(\y_n^{ijk})_n$ of $\H^{ijk}$ for the
  Hermitian form.  Define $\gamma_n^{ijk}=(\y_n^{ijk}\mid\y_n^{ijk})$.
  Let $\y_n^{k^*j^*i^*}$ be the dual basis of $\H^{k^*j^*i^*}$ for the
  pairing $\Theta$.  The bases of
  $\H^{jki},\H^{kij},\H^{j^*i^*k^*},\H^{i^*k^*j^*}$ are then defined
  using the mappings $\rot$ so that \eqref{eq:rot-basis} is satisfied
  (the composition of 3 $\rot$ maps is the identity).  We claim that
  all these bases are orthogonal for the Hermitian forms.  To show
  this for $(\y_n^{k^*j^*i^*})_n$, we introduce the semilinear
  bijections $\ddagger:\H^{ijk}\to \H^{k^*j^*i^*}$ defined by
  $y\mapsto y^\ddagger=(w_{k^*}^{-1}\otimes w_{j^*}^{-1}\otimes
  w_{i^*}^{-1})\bp{y^\dagger}^*$ which satisfy the property that for any
  $y_1,y_2\in\H^{ijk},$
  $$\bp{{y_1}^\ddagger}^\ddagger=\beta(i)^{-1}\beta(j)^{-1}\beta(k)^{-1}y_1\text{
    and }(y_1|y_2)=\Theta_{ijk}(y_1^\ddagger,y_2).$$
  Now, since $(\y_n^{ijk})_n$ is an orthogonal basis,
  $\bp{(\gamma_n^{ijk})^{-1}(\y_n^{ijk})^\ddagger}_n$ is the corresponding dual
  basis of $\H^{k^*j^*i^*}$.  So
  $(\y_n^{ijk})^\ddagger=\gamma_n^{ijk}\y_n^{k^*j^*i^*}$ and it is
  enough to show that $(\y_n^{ijk})^\ddagger$ is orthogonal.  This is true because
  $$\bp{(\y_n^{ijk})^\ddagger\mid(\y_m^{ijk})^\ddagger}=\Theta_{k^*j^*i^*}\bp{(\y_n^{ijk})^{\ddagger\ddagger},(\y_m^{ijk})^\ddagger}=\frac{\Theta_{ijk}\bp{(\y_m^{ijk})^\ddagger,\y_n^{ijk}}}{\beta(i)\beta(j)\beta(k)}=\frac{\bp{\y_m^{ijk}\mid\y_n^{ijk}}}{\beta(i)\beta(j)\beta(k)} . $$
\end{proof}


\subsection{Invariant of labeled uni-trivalent graphs}

Let $\cat$ be a relative $\Gr$-Hermitian-spherical category equipped with basic data as in Section~\ref{sec:basis}.  A $\cat$-labeled uni-trivalent graphs is an
unoriented graph $\Gamma$ embedded in the unit disc along with
\begin{itemize}
  \item a map $\sigma_1$: oriented edges$\to I$ with $\sigma_1(-e)=\sigma_1(e)^*$, and
  \item a map $\sigma_0$: trivalent vertices$\to \N$ such that if all edges labeled $(i,j,k)$ of a trivalent vertex $v$ are oriented inward toward the vertex, then $\sigma_0(v) \in [1,\delta_{ijk}]$.
\end{itemize}

To each $\cat$-labeled uni-trivalent graph we can associate a $\cat$-colored ribbon graph  $\tilde{\Gamma}$ by
\begin{itemize}
\item add a bivalent vertex $v_{\epsilon}$ that bisects each
  unoriented edge $\epsilon$ and orienting the two resulting edges
  $\overleftarrow{\epsilon}$ and $\overrightarrow{\epsilon}$ away from
  the vertex $v_{\epsilon}$
\item label each bivalent vertex $v_{\epsilon}$ with a coupon labeled
  $w_{\sigma_1(\vec{\epsilon})}$, which is well defined by
  \eqref{eq-w_i} and the equalities
  $\sigma_1(\overrightarrow{\epsilon}) =
  \sigma_1(-\overleftarrow{\epsilon}) =
  \sigma_1(\overleftarrow{\epsilon})^{\ast}$.
\item replace each resulting trivalent vertex $v$ with inward oriented
  edges $(i,j,k)$ by a coupon labeled $\y_{\sigma_0(v)}^{ijk}$.
\end{itemize}
%
%
Following the unit circle clockwise we see the univalent vertices of
$\wt\Gamma$   form  a word $((V_{i_1},+),\cdots, (V_{i_n},+))$.
%
To each such $\cat$-colored ribbon graph $\tilde{\Gamma}$, the
RT-functor produces an element
$\Gfun(\wt\Gamma)\in\Hom_\cat(\unit,V_{i_1}\otimes\cdots\otimes
V_{i_n})$.  We can then define an operation on labeled uni-valent
graphs via
\begin{equation}\label{eq:label-uni-eval1}
  \brk{\Gamma}:=\Gfun(\wt\Gamma),
\end{equation}
that
only depends on $\Gamma$.
 For a closed (without univalent vertices)
graph $\Gamma$,
we modifiy the definition of $\brk\Gamma$ to
\begin{equation} \label{eq:label-uni-eval}
  \brk\Gamma ':= \Gfun'(\tilde{\Gamma})
\end{equation}
using a cutting presentation and the modification from \eqref{eq:defG'}.
Since $\cat$ is relative $\Gr$-spherical, then the graph evaluation
\eqref{eq:label-uni-eval} extends to a well-defined evaluation for
spherical ribbon graphs, i.e. $\cat$-ribbon graphs in $S^2$,
see~\cite[Lemma 9]{GP3}.

\subsection{Conjugations for labeled uni-trivalent graphs}
Now suppose that $\cat$ is a relative $\Gr$-Hermitian-spherical
category equipped with basic data as in Section~\ref{sec:basis}.
There is a non-involutive semilinear conjugation operator
$\Gamma\mapsto \Gamma^\ddagger$ on the space generated by labeled
planar uni-trivalent networks as follows. The definition
of this operator is inspired by the $\dagger$ functor and Equation
\eqref{eq:label-uni-eval1}. The image of a graph $\Gamma$ with state
$\sigma$, with set of trivalent vertices $\Gamma_0$, with set of
univalent vertices $\Gamma'_0$ and with set of edges $\Gamma_1$ is
given by
\begin{equation}
  \label{eq:plan-dag}
  \Gamma^\ddagger=c(\Gamma)\mirror \Gamma
\end{equation}
where $\mirror \Gamma$ is the mirror image of $\Gamma$  with the dual state and the complex coefficient $c(\Gamma)$ is given by
$$c(\Gamma)=\prod_{v\in \Gamma_0}\gamma(\sigma(v))\prod_{u\in \Gamma'_0}\beta(\sigma_1(u))^{-1}\prod_{e\in \Gamma_1}\beta(\sigma_1(e))$$
where $\sigma_1(u)$ is $\sigma_1$ of the unique edge adjacent to $u$
and
$\gamma(\sigma(v))=\gamma^{\sigma_1(e_1)\sigma_1(e_2)\sigma_1(e_3)}_{\sigma_0(v)}$
where $e_1,e_2,e_3$ are the $3$ ordered edges adjacent to $v$ oriented
toward $v$.  A network with no univalent vertices is closed.
\begin{remark}
Note that for an oriented edge $e$, $\beta(\sigma(-e))=\beta(\sigma(e))$.  Thus $\beta(\sigma(e))$ make sense even if we forget the orientation of $e$.
\end{remark}
\begin{proposition}Let $\cat$ be a relative Hermitian-pivotal (resp. relative Hermitian-spherical) category,
then for $\Gamma$  a closed planar (or spherical) network, the
  evaluation of the network $\Gamma^\ddagger$ is the conjugate of the
  value of $\Gamma$
  \begin{equation}
    \label{eq:ddagger}
    \brk{\Gamma^\ddagger} '=\overline{\brk{\Gamma} '}.
  \end{equation}
\end{proposition}

\subsection{Scalar 6j-symbols in relative Hermitian-spherical categories}
Now assume that $\cat$ is relative Hermitian-spherical equipped with basic data.    We define scalar $6j$ symbols from  the tetrahedral graph ribbon graph $\tilde{\Gamma}$
\begin{equation} \label{eq:mod6j}
N^{j_1 j_2 j_3}_{j_4 j_5 j_6}(^{a_1a_2}_{a_3a_4})=\Gfun'\left({\hackcenter{\begin{tikzpicture}[   decoration={markings, mark=at position 0.6 with {\arrow{>}};}, scale =0.7]
        \draw[thick,  postaction={decorate}, out=-30, in=100] (-.7,3) to (0,2);
        \draw[thick,  postaction={decorate}] (-.7,3) .. controls ++(-.35,-.5) and ++(-.35,.5) .. (-.7,1);
        \draw[thick,  postaction={decorate}, out=-70, in=140] (-.7,1) to (0,0);
        \draw[thick,  postaction={decorate}] (0,2) -- (-.7,1);
        \draw[thick,  postaction={decorate}] (0,2) .. controls ++(.35,-.5) and ++(.35,.5) ..  (0,0);
        \draw[thick, ] (0,0) .. controls ++(0,-.5) and ++(0,-.5) ..  (1,0);
        \draw[thick,  postaction={decorate}] (1,0) to (1,3);
        \draw[thick, ] (-.7,3) .. controls ++(0,.75) and ++(0,.55) ..  (1,3);
        \draw[gray,fill=gray] (-.7,1) circle (.5ex);
        \draw[gray,fill=gray](0,2) circle (.5ex);
        \draw[gray,fill=gray] (-.7,3) circle (.5ex);
        \draw[gray,fill=gray] (0,0) circle (.5ex);
        \node at (-1.2,2.3) {$\scs j_1$};
        \node at (-.35,1.9) {$\scs j_2$};
        \node at (-.85,.55) {$\scs j_3$};
        \node at (.5,1.2) {$\scs j_4$};
        \node at (.2,2.5) {$\scs j_6$};
        \node at (1.3,2) {$\scs j_5$};
      \end{tikzpicture}}} \right)
\end{equation}
 where the coupons at vertices are filled with
$$\y_{a_1}^{j_1 j_2 j_3^*},\quad \y_{a_2}^{j_3j_4j_5^*},\quad \y_{a_3}^{j_5j_6^*j_1^*},\quad \y_{a_4}^{j_6j_4^*j_2^*} .$$

\begin{proposition}  \label{prop:prop6j}
The $6j$ symbols associated with a relative Hermitian-spherical category $\cat$ satisfy the equalities \eqref{eq:theta} and  \eqref{eq:symm} --  \eqref{eq:6j}.
\end{proposition}

\begin{proof}
The definition of modified $6j$ symbols in the relative spherical category $\cat$ implies they have the symmetries of an oriented tetrahedron implying \eqref{eq:symm}.   Equation \eqref{eq:pent} follows from \cite[Theorem 11]{GP3} and \eqref{eq-ortho} from \cite[Theorem 12]{GP3}.  Finally, \eqref{eq:theta} and \eqref{eq:6j} follow as in the proof of \cite[Appendix A.15]{GLPS2}.
\end{proof}

\subsection{Relative $\Gr$-LW-data from relative Hermitian-spherical}
Finally, we explain how the data of an $(\X,\qd)$-relative $\Gr$-Hermitian-spherical category provides $(\X,\qd)$-relative $\Gr$-spherical LW-data in the sense of Section~\ref{subsec:basicinput}.
\begin{theorem} \label{thm:sphtodata}
Let $\cat$ be a $(\X,\qd)$-relative Hermitian-spherical category where $\X$ contains all elements of $\Gr$ of order $2$ and $3$.  Then $\cat$ determines relative $LW$-data
\[
(\Gr, \X, (I_g)_{g\in\Gr}, \delta, \mb, \md,N,\beta,\gamma)
\]
 where $I_g$ is as in \eqref{eq:Ig}, $\delta$ is defined in \eqref{eq:def-delta},  $\mb, \md$ come from the relative spherical data, $N$ are the modified 6$j$ symbols from \eqref{eq:mod6j},    $\beta$ are defined in \eqref{def:beta},   and $\gamma$ are defined in the proof of Proposition~\ref{prop:Herm-basicdata}.
 \end{theorem}

\begin{proof}
This is immediate from the definitions and Proposition~\ref{prop:prop6j}.
\end{proof}

There is a modified Turaev-Viro TQFT associated to the data of an $(\X,\qd)$-relative $\Gr$-spherical category~\cite{GP3}.  This  is a $(2+1)$-dimensional TQFT defined on
a 3-manifold $M$ with the additional data of an isotopy class of an
embedded graph $Y$ in $M$ and a cohomology class
$[\col]\in H^1(M,\Gr)$.  Given a triangulation $\cal{T}_M$ of $M$, the
graph $Y$ is assumed to be Hamiltonian, meaning that it intersects
every vertex of the triangulation exactly once, and its edges are part
of the triangulation.  When the manifold $M$ has boundary, the triangulation of $M$ induces a
triangulation $\cal{T}$ on the boundary $\Sigma$ and the cohomology
class $[\col]\in H^1(M,\Gr)$ restricts
to $[\col_{|\Sigma}]\in H^1(\Sigma,\Gr)$.
In \cite{GP3}, it is assumed the graph $Y$ is embedded transversely to
the boundary surface.  Thus, the modified Turaev-Viro invariant for
the surface $\Sigma$ includes the additional data of a finite set of
marked points ${\mathsf m} \subset \Sigma$ on the surface
corresponding to where the graph intersects the surface.

Assume that the set of vertices in the triangulation $\Sigma$ is exactly the set of marked points
As explained in \cite[Section F]{GLPS2}, we have the following.
\begin{theorem} \label{thm:ground=TV}
  The degenerate ground state of the Hamiltonian in \eqref{eq:Hamiltonian0} defined on a surface $\Sigma$ is isomorphic to the vector space assigned to the surface by the modified TQFT from \cite{GP3} defined from the corresponding relative $\Gr$-Hermitian-spherical category:
  $$\operatorname{Ker} H\simeq TV(\Sigma,{\mathsf m},[\col])$$
  where $[\col]\in H^1(\Sigma,\Gr)$ and the cardinality of
  ${\mathsf m}$ is the number of vertices in the triangulation $\cal{T}$ of $\Sigma$.

 In particular, the ground state of the system is a topological invariant of the surface $\Sigma$ equipped with the set of points $\m$ and the cohomology class $[\col]\in H^1(\Sigma,\Gr)$.
\end{theorem}

\section{Hermitian surgery TQFT} \label{sec:hermribbon}
We review a WRT-style TQFT having a Hermitian structure.

\subsection{Hermitian structure on decorated cobordisms}
\newcommand{\ve}{\varepsilon}
\newcommand{\Cob}{{\mathcal{C}ob}}
Let $\cat$ be a relative $\Gr$-Hermitian-modular category.
We follow Turaev \cite[I.II.5.1]{Tu} and define the dagger of a $\cat$-colored
ribbon graph $T$ as the following transformation: invert the
orientation of the surface of $T$, reverse directions of bands
and annuli of $T$, exchange bottom and top bases of coupons, and
replace the colors $f$ of coupons with $f^\dagger$.  The colors of
edges do not change.

The boundary of a $\cat$-colored ribbon graph is a set of
$\cat$-colored framed points where a framed point $p=(V,\ve)$ is a
point equipped with a sign $\ve$, a non-zero vector (its framing) tangent to the surface $T$ and to the direction of the band, and
a color which is an object $V\in \cat$.   Let $\Gfun(p)=V$ if
$\ve=+$ and $\Gfun(p)=V^*$ if $\ve=-$, where $\Gfun$ is the Reshetikhin-Turaev functor from Section \ref{subsec:renorm-ribbon-graph}. The conjugate $\wb p$ of a
$\cat$-colored framed point $p$ is obtained by changing the sign
and framing to their opposites.

Recall the category of decorated cobordisms first introduced in \cite{D17}, for a short description also see \cite[Section 1.3]{DGP2}.
\begin{definition}[Objects of $\Cob$]
  A \emph{decorated surface} is a $4$-tuple $\dSu =(\Su,\{p_i\},\coh,
  {\La})$ where:
\begin{itemize}
\item $\Su$ is a closed, oriented surface which is
  an ordered disjoint union of connected surfaces each having a distinguished base point $*$;
\item $\{p_i\}$ is a finite (possibly empty) set of framed points colored with 
  objects of $\cat$
  whose framings are tangent to the surface $\Su$;
\item $\coh\in H^1(\Su\setminus \{p_1,\ldots, p_k\}, *;\Gr)$ is a
  cohomology class;
\item {\em compatibility condition}: letting $\coh( m_i)=a_i\in\Gr\setminus \X$
  where $m_i$ is a positively oriented circle around $p_i$, then we
  require $F(p_i)\in\cat_{a_i}$;
\item ${\La}$ is a Lagrangian subspace of $H_1(\Su;\R)$.
\end{itemize}
A decorated surface is \emph{admissible} if each component of $\dSu$
either has a framed point colored by a projective object or it
contains a closed curve $\gamma$ such that $\coh(\gamma)\notin\X$.
Objects of $\Cob$ are admissible decorated surfaces.
\end{definition}

\begin{definition}[Morphisms of $\Cob$]\label{D:Cobordisms}
  Let ${\dSu}_\pm=(\Su_\pm,\{p_i^\pm \}, \coh_\pm, \La_\pm)$ be admissible,
  decorated surfaces.  A \emph{decorated cobordism} from ${\dSu}_-$ to
  ${\dSu}_+$ is a $5$-tuple $\dM=(M,T,f, \coh,n)$ where:
\begin{itemize}
\item $M$ is an oriented $3$-manifold with boundary $\partial M$;
\item $f: \overline{\Su_-}\sqcup \Su_+\to \partial M$ is a
  diffeomorphism preserving the orientation, and the image under
  $f$ of the base points of $\overline{\Su_-}\sqcup \Su_+$ is denoted by $*$;
\item $T$ is a $\cat$-colored ribbon graph in $M$ such that
  $\partial T=\{\wb{f(p_i^-)}\}\cup \{f(p_i^+)\}$;
\item $\coh\in H^1(M\setminus T,*;\Gr)$ is a cohomology class
  relative to the base points on $\partial M$, such that the
  restriction of $\coh$ to
  $(\partial M\setminus \partial T)\cap \Su_\pm$ is
  $(f^{-1})^*(\coh_{\pm})$;
\item the coloring of $T$ is compatible with $\coh$, i.e. each
  oriented edge $e$ of $T$ is colored by an object in
  $\cat_{\coh(m_e)}$ where $m_e$ is the oriented meridian of $e$;
  \item all the edges of $T$ are colored by Hermitian objects of $\cat$;
\item $n$ is an arbitrary integer
  called the {\em signature-defect} of $\dM$.
\end{itemize}
We can summarize the first four items by saying that
$\partial\dM=\dSu_-^*\sqcup \dSu_+$ where the dual of a decorated
surface $\dSu =(\Su,\{p_i\},\coh, {\La})$ is defined to be
$\dSu^*=(\wb\Su ,\wb{\{p_i\}},\coh,{\La})$.
A decorated cobordism is {\em admissible} if each component of $M$ contains
either a component of $T$ with an edge colored by a projective object or it contains a
closed curve $\gamma$ such that $\coh(\gamma)\notin\X$.  This
condition is automatically satisfied by components of $\dM$ with
non-empty boundary.  Hence
morphisms of $\Cob$ are orientation preserving diffeomorphism classes of
admissible decorated cobordisms.
\end{definition}

\begin{proposition}
  Let $\dM=(M,T,f, \coh,n):\dSu_-\to\dSu_+$ be a cobordism in $\Cob$.
  Then the following defines a cobordism $\dM^\dagger:\dSu_+\to\dSu_-$
  in $\Cob$:
  $$\dM^\dagger=(\wb M,T^\dagger,\wb f, \coh,-n),$$
  where $\wb M$ is $M$ with opposite orientation, and
  $\wb f=\overline{\Su_+}\sqcup \Su_-\to \partial \wb M$ is the same
  set-theoretic map as $f$.
  Furthermore, with the above assignment, $\Cob$ is a Hermitian ribbon
  category.
\end{proposition}
\begin{proof}
The proof is the same as the proof of \cite[Proposition 5.4]{GLPS1}.
\end{proof}



\subsection{The $3$-manifold invariant}
Let $\cat$ be a relative $\Gr$-Hermitian-modular category with stabilization coefficients $\Delta_{\pm\Omega}$.  If $T_{-\Omega_g}^V$ and $T_{+\Omega_g}^V$ are the graphs on the left side of the equalities in Figure \ref{F:stabilization_coefficients_Omega} (here the open strand is colored with $V$) then
$$
\Delta_{\pm\Omega}\Id_V=\Gfun(T_{\pm\Omega_g}^{V})=\Gfun((T_{\mp\Omega_{-g}}^{V^*})^\dagger)=\Gfun(T_{\mp\Omega_{-g}}^{V^*})^\dagger=(\Delta_{\mp\Omega}\Id_{V^*})^\dagger=\wb{\Delta_{\mp\Omega}}\Id_V.
$$
Thus, for a relative $\Gr$-Hermitian-modular category we have
$\Delta_{\pm\Omega}=\wb{\Delta_{\mp\Omega}}$ .

Fix a choice of a
square root $\calD_{\Omega} \in \Bbbk$ of
$\Delta_{- \Omega} \Delta_{+ \Omega}$ (this might require one to
replace $\Bbbk$ with a quadratic extension) and set
$\delta_{\Omega}=\frac{\calD_{\Omega}}{\Delta_{-
    \Omega}}=\frac{\Delta_{+\Omega}}{\calD_{\Omega}}$.  Notice that
since $\Delta_{\pm\Omega}=\wb{\Delta_{\mp\Omega}}$, we are able to
deduce that $\wb\calD_{\Omega}=\calD_{\Omega}$ and
$\wb{\delta_{\Omega}}=\delta_{\Omega}^{-1} $.

If $\dM=(M,T,f, \coh,n)$ is a closed decorated manifold (i.e
$\dM\in\End_\Cob(\emptyset)$) which is connected, a surgery
presentation of $\dM$ is a $\cat$-colored ribbon graph
$T\cup L\subset S^3$ such that $M$ is obtained from $S^3$ by surgery
on $L$ and each component of $L$ is colored by a Kirby color
$\Omega_g$ of
degree $g$ equal to the value of the cohomology class on its meridian.
Then the invariant of $\dM$ is given by
\begin{equation}
  \label{eq:Zrconnected}
  \Zr_\cat(\dM)= \calD_{\Omega}^{-1-\ellr} \delta_{\Omega}^{-\sigma(L)+n}\Gfun'(L\cup T)
\end{equation}
where $l\in\N$ is the number of components of $L$ and $\sigma\in\Z$ is the
signature of the linking matrix $\lk(L)$.   The invariant is extended multiplicatively for disjoint unions.

\begin{lemma}  \label{lem:dagofZ}
Let $\dM$ be a closed decorated manifold then
\[
\Zr_\cat (\dM^\dagger) =
\overline{\Zr_\cat(\dM)}
\]
\end{lemma}

\begin{proof}
  If a framed link $L$ gives rise via surgery to the manifold $M$,
  then the manifold $\overline{M}$ may be constructed from
  $\overline{L}$ where $\overline{L}$ is the mirror image of $L$.  It
  follows from the conjugation on the underlying Hermitian category
  that $\Gfun'(\overline{L}\cup T^\dagger)=\overline{\Gfun'(L\cup T)}$.

  If the linking matrix of $L$ has signature $\sigma$, then the
  signature of the linking matrix of $\overline{L}$ is $-\sigma$.  Thus,
  $$\Zr_\cat (\dM^\dagger) =
  \Zr_\cat (\wb M,T^\dagger,\wb f, \coh,-n)
  =\calD_{\Omega}^{-1-\ellr} \delta_{\Omega}^{-\sigma(\wb L)-n}\Gfun'(\overline{L}\cup T^\dagger)
    =\wb{\calD_{\Omega}}^{-1-\ellr} \wb{\delta_{\Omega}}^{-\sigma(L)+n}\overline{\Gfun'(L\cup T)}
=
\overline{\Zr_\cat(\dM)}
$$
where the third equality follows from facts that $\calD_{\Omega}$ is real and $\wb{\delta_{\Omega}}=\delta_{\Omega}^{-1}$.
\end{proof}

\begin{corollary} \label{cor:Hersymcob}
  The pairing $\Cob(\emptyset,\dSu)\times\Cob(\emptyset,\dSu)\to\C$, given by
  $$(\dM_1,\dM_2)\mapsto \Zr_\cat (\dM_1^\dagger\circ\dM_2)\in\C$$
  has Hermitian symmetry.
\end{corollary}

\begin{proof}
This follows directly from Lemma \ref{lem:dagofZ}.
\end{proof}

\subsection{The $(2+1)$-TQFT}
Fix a choice $g_0 \in \Gr\setminus \X$ and a generic Hermitian simple object $V_{g_0}$ in $\cat_{g_0}$.
For $k\in \Z$, let $\hS_k$ be the decorated sphere
\[
 \hS_{k} = (S^2,\{(V_{g_0},1),(\sigma(k),1),(V_{g_0},-1)\},\coh,\{ 0 \})
\]
of $\Cob$ which is determined by the framed colored points $\{(V_{g_0},1),(\sigma(k),1),(V_{g_0},-1)\}$ in $S^2$.
Now we define the state space associated to a decorated surface in the following way:
$$\V(\dSu)=\Span_\C\qn{\Hom_{\Cob}(\emptyset,\dSu)}/K_{\dSu}$$
where $K_{\dSu}$ is the right kernel of the bilinear pairing given on
generators
$$
\begin{array}{ccl}
  \Span_\C\qn{{\Cob}(\dSu,\emptyset)}\otimes
  \Span_\C\qn{{\Cob}(\emptyset,\dSu)}&\to&\C\\ {}
  [{\dM_1}]\otimes[\dM_2]&\mapsto& \Zr_\cat (\dM_1\circ\dM_2)
\end{array}
$$
$$\VV(\dSu)=\bigoplus_{k\in\Z}\VV_k(\dSu)\text{ where }\VV_k(\dSu)=\V(\dSu\sqcup\hS_k) .$$

These state spaces are part of a $(2+1)$-TQFT constructed in  \cite{D17}.  Since $\dagger:{\Cob}(\emptyset,\dSu)\to {\Cob}(\dSu,\emptyset)$ is bijective, the pairing described in Corollary \ref{cor:Hersymcob} descends to a non-degenerate Hermitian pairing on the state spaces $ \VV(\dSu)$ of the TQFT.

\begin{theorem} \label{thm:hermWRT}
The TQFT $(\VV,\Zr_\cat) $ is Hermitian.
More specifically, for any decorated surface $\dSu$, there is a non-degenerate Hermitian pairing
\[
\langle \cdot, \cdot \rangle_{\VV(\dSu)}
\]
and
for any $y \in \VV(\dSu_-)$ and for any $x \in \VV(\dSu_+)$, and
for any decorated cobordism $\dM : \dSu_- \to \dSu_+$ between decorated surfaces, there is an equality
\begin{equation} \label{adjofsurface}
\brk{ x,\VV(\dM)(y) }_{\VV(\dSu_+)}
=
\brk{  \VV(\dM^\dagger)(x),y }_{\VV(\dSu_-)}
.
\end{equation}
\end{theorem}

\begin{proof}
The fact that the pairing is non-degenerate and Hermitian follows from the discussion above.

The equality \eqref{adjofsurface} follows from the functoriality of the TQFT and Lemma \ref{lem:dagofZ}.
The details follow as in \cite[Theorem III.5.3]{Tu}.
\end{proof}



\section{Examples from quantum (super)groups} \label{sec:examples}
In this section we provide examples of  $\Gr$-LW-systems coming from quantum groups and super quantum groups.

\subsection{Complex finite-dimensional Lie algebras}
Let $\mathfrak{g}$ be a simple finite-dimensional complex Lie algebra.
Let $r$ be an odd integer such that $r \geq 3$, (and
$ r \notin 3 \mathbb{Z}$ if $\mathfrak{g}=G_2$).

Fix a set of simple roots $\{ \alpha_1, \ldots, \alpha_n \}$ of
$\mathfrak{g}$.  Let $\Delta$ and the $\Delta^+$ be the corresponding
sets of roots and positive roots respectively.  Let $A=(a_{ij})$ be
the Cartan matrix corresponding to the set of simple roots.  Let
$D=(d_1,\ldots,d_n)$ be a diagonal matrix such that $DA$ is symmetric
and positive definite.  Let $\mathfrak{h}$ be the Cartan subalgebra of
$\mathfrak{g}$ spanned by elements $H_1, \ldots, H_n$ where
$\alpha_i(H_j)=a_{ji}$.  Denote by $L_R$ (the root lattice), the
$\mathbb{Z}$-module generated by the simple roots $\{\alpha_i\}$.  Let
$\langle , \rangle$ be the form on $L_R$ determined by
$\langle \alpha_i, \alpha_j \rangle = d_i a_{ij}$.  The weight lattice
$L_W$ is the $\mathbb{Z}$-module generated by the elements of
$\mathfrak{h}^*$ which are dual to the elements $\{H_1,\ldots,H_n\}$.
Let $L'_R=\{\beta\in\mathfrak{h}^*:\brk{\beta,L_R}\subset\Z\}$ be the lattice
dual to $L_R$.
We have
$L_R\subset L_W\subset L'_R\subset\mathfrak{h}^*$. Let
$\rho=\frac{1}{2} \sum_{\alpha \in \Delta^+} \alpha \in L_W$.

Let $q=e^{\frac{2\pi i }{r}}$ and for $i=1,\ldots,n$ set $q_i=q^{d_i}$.

The Drinfeld-Jimbo quantum group $\mathcal{U}$ associated to $\mathfrak{g}$ is the $\mathbb{C}$-algebra generated by
$K_{\beta}, X_i, X_{-i}$ for $\beta \in L_W$ and $i=1,\ldots,n$ subject to relations:
\begin{equation}
\label{DJ1}K_0=1, \quad   K_{\beta}K_{\gamma}=K_{\beta+\gamma}, \quad  K_{\beta} X_{\pm i} K_{-\beta} = q^{\pm \langle \beta, \alpha_i \rangle} X_{\pm i},
\end{equation}
\begin{equation}
\label{DJ2}
[X_i, X_{-j}]= \delta_{ij} \frac{K_{\alpha_i}- K_{\alpha_i}^{-1}}{q_i-q_i^{-1}}, \quad
\sum_{k=0}^{1-a_{ij}} (-1)^k {1-a_{ij} \brack k}_{q_i} X_{\pm i}^j X_{\pm j} X_{\pm i}^{1-a_{ij}-k} = 0, \text{ if } i \neq j.
\end{equation}

The quantum group $\mathcal{U}$ can be given the structure of a Hopf algebra with coproduct $\Delta$, counit $\epsilon$ and antipode $S$ defined by
\begin{align}
\label{DJH1}\Delta (X_i)&=X_i \otimes K_{\alpha_i} + 1 \otimes X_i, & \epsilon(X_i)&=0, & S(X_i)&=-X_i K_{\alpha_i}^{-1},\\
\label{DJH2} \Delta(X_{-i})&=X_{-i} \otimes 1+K_{\alpha_i}^{-1} \otimes X_{-i}, & \epsilon(X_{-i})&=0, & S(X_{-i})&=-K_{\alpha_i} X_{-i},\\
\label{DJH3}\Delta(K_{\beta})&=K_{\beta} \otimes K_{\beta} , & \epsilon(K_{\beta})&=1 , & S(K_{\beta})&=K_{-\beta}.
\end{align}

The unrolled quantum group $\mathcal{U}^H$ associated to $\mathfrak{g}$ is the $\mathbb{C}$-algebra generated by
$K_{\beta}, X_i, X_{-i}, H_i$ for $\beta \in L_W$ and $i=1,\ldots,n$ subject to relations \eqref{DJ1}, \eqref{DJ2}, and:
\begin{equation}
\label{UR1}
[H_i, X_{\pm j}]=\pm a_{ij} X_{\pm j}, \quad
[H_i, H_j]=[H_i, K_{\beta}]=0.
\end{equation}
The algebra $\U^H$ is also a Hopf algebra with coproduct, counit, and antipode defined by \eqref{DJH1}, \eqref{DJH2}, \eqref{DJH3}, and
\begin{equation}
\label{UH1}
\Delta(H_i)= H_i \otimes 1 + 1 \otimes H_i, \quad \epsilon(H_i)=0, \quad S(H_i)=-H_i.
\end{equation}
The Hopf algebra $\U^H$ is pivotal with pivot $g=K_{2\rho}^{1-r}$.

Let $\alpha_1, \ldots, \alpha_n$ be any ordering of the simple roots and let $s_{i_1} \cdots s_{i_N}$ be a reduced decomposition of the longest element of the Weyl group.
Then
$
\beta_1=\alpha_1, \beta_2=s_{i_1} \alpha_{i_2}, \ldots, \beta_N=s_{i_1} \cdots s_{i_{N-1}} \alpha_{i_N}
$
is a total ordering of the set of positive roots.
We call $\beta_{*}=(\beta_1, \ldots, \beta_N)$ the convex order of $\Delta^+$.
Let $X_{\pm \beta_i}$ be the corresponding positive and negative root vectors respectively (see \cite[Definition 8.1.4]{CP}).
Denote by $\mathcal{Z}^0$ the commutative subalgebra
\[
\mathcal{Z}^0 = \mathbb{C} \langle X_{\pm \beta}^{r}, K_{\gamma}^r  \mid  \beta \in \beta_*, \gamma \in L_W  \rangle  .
\]

Let
$I=\operatorname{Span}_\C\bs{X_{\pm \beta}^{r}:\beta\in\beta_{*}}$.
$I$ generates ideals in the algebras $\U$ and $\U^H$ which are Hopf
ideals.  Correspondingly, we call $\wb \U$ and $\wb \U^H$ the quotient Hopf algebras.


A weight is a $\mathbb{C}$-algebra homomorphism $\lambda$
from the $\mathbb{C}$-algebra generated by the $H_{i}$ to
$\mathbb{C}$.  For a $\wb\U^H$-module $M$, the weight space
$M_{\lambda}$ is the subspace of $M$ where the $H_i$ act as a scalar
via $\lambda$ and if $\beta=\sum_i b_i \alpha_i$, then $K_{\beta}$
acts as the scalar
$\prod_i q_i^{b_i \lambda(H_i)}=q^{\langle \beta, \lambda \rangle} $.
A $\wb\U^H$-module $M$ is called a weight module if it is equal
to a direct sum of its weight spaces.

Let $\mathcal{D}$  be the category of
finite-dimensional weight modules over
$\wb\U^H$.
Let $\Gr$ be the group $\mathfrak{h}^*_\R/L_R\cong (\Unitary)^n$ where $\mathfrak{h}^*_\R=L_W\otimes_\Z\R$ is the set of real weights and $\Unitary \simeq S^1$ is the group of $1\times1$ unitary matrices.
For $g\in \Gr$, let $\mathcal{D}_{g}$ be the full subcategory of
$\mathcal{D}$ whose objects are the $\wb\U^H$-weight modules
whose weights are all in the class $g \;(\text{mod } L_R)$.
%
%
 Define
\begin{equation}
\exp_q^<(x) = \sum_{0 \leq n<r} \frac{q^{n(n-1)/2}}{[n]!} x^n\in\C[x]
\end{equation}

\begin{lemma}
We have the identity
$ \exp_q^<(x) \exp^<_{q^{-1}}(-x)=1$ modulo $x^r$.
\end{lemma}

For a root $\beta$, let $q_{\beta}=q^{\frac{\langle \beta, \beta \rangle}{2}}$.
\begin{equation}
  \check{R}=\prod_{\beta \in \beta_{*}} \exp^<_{q_{\beta}}((q_{\beta} - q_{\beta}^{-1}) X_{\beta} \otimes X_{-\beta})
\end{equation}
The following elements define operators on the category $\mathcal{D}$, (alternatively, they are elements of the nuclear completion of $\wb\U^H$ see \cite{GHP})
\begin{equation}
\mathcal{H} = q^{\sum_{i,j} d_i (A^{-1})_{ij} H_i \otimes H_j} , \quad \quad R = \mathcal{H} \check{R} .
\end{equation}

Consider the operator $u_{\mathcal{H}} = q^{\sum_{i,j} -d_i(A^{-1})_{ij} H_i H_j } $.
Note that $S(u_{\mathcal{H}})= u_{\mathcal{H}} $ and if $M$ is a $\wb\U^H$-module and $m \in M$ is an element of weight $\lambda$,
then $u_{\mathcal{H}} m = q^{- \langle \lambda, \lambda \rangle} m$.
Also note that if $x \in \wb\U^H$ has weight $\alpha$, then
$ u_{\mathcal{H}}  x = q^{- \langle \alpha, \alpha \rangle}x K^2_{-\alpha} u_{\mathcal{H}}  $.
Then conjugation $Ad_{u_{\mathcal{H}}^{-1}}$ by $u_{\mathcal{H}}^{-1}$ is an automorphism of $\mathbb{C}[H_1,\ldots, H_n]$.
Let $u = \mu \circ ((Ad_{u_{\mathcal{H}}^{-1}} \circ S) \otimes \Id)(\check{R}_{21}) $ where $\mu $ is multiplication.
Let $\nu$ and $\theta$ be the operators defined by
$\nu = u_{\mathcal{H}}  u$ and $\theta= \nu K^{r-1}_{2 \rho}=\nu g^{-1}$.  We now recall the following key result.

\begin{proposition} \cite[Theorem 19]{GP4} \label{prop:DHribbon}
The category $\mathcal{D}$ is a $\mathbb{C}$-linear ribbon category with twist
$\theta_M \colon M \rightarrow M, m \mapsto \theta^{-1}(m)$.
\end{proposition}

\begin{lemma}\label{L:antiautomorphismUHg}
There is an antiautomorphism $\dagger$ of $\wb\U^H$ (as in Equation \eqref{eq:dag2}) given by
\begin{equation}
\dagger(E_i)=F_i, \;\; \dagger(F_i)=E_i, \;\; \dagger(K_{\mu}) = K_{-\mu}, \;\; \dagger(H_{i}) = H_{i}.
\end{equation}
We will write $x^{\dagger}$ in place of $\dagger(x)$ when no confusion is likely to arise.
\end{lemma}

\begin{proof}
This is a straightforward verification.
\end{proof}

\begin{lemma} \label{LAbraidherm}
  For all finite-dimensional simple complex Lie algebras, one has the identities:
  \begin{equation}
    (\dagger\otimes\dagger)\bp{R}=\tau(R^{-1})\et\dagger(g)=g^{-1}.
  \end{equation}
  As a consequence, $\dagger(\theta)=\theta^{-1}$.
\end{lemma}

 \begin{proof}
  First, note that
   $ \left( \exp^<_{q_{\beta_i}}(X) \right)^{-1}=
   \exp^<_{q^{-1}_{\beta_i}}(-X) $ and $g=K_{2 \rho}^{1-r}$.
   It is straightforward to check that $\dagger(X_{\beta})=X_{-\beta}$. One way to see this is to use the fact that $\dagger$ commutes with Lusztig's braid operators.
   The first result now follows.

It is clear that $\dagger(g)=g^{-1}$ and we note further that the third equality is
   a consequence of the two first.
\end{proof}
Let  $\catdSeq$ be the category of sesquilinear  weight modules over $\wb\U^H$ with real weights (see Equation \eqref{E:sesquilinearComp}).
%
Recall a sesquilinear module $W$ is Hermitian if its sesquilinear form has Hermitian symmetry: if for any $v_1,v_2\in W$, $(v_2|v_1)=\wb{(v_1|v_2)}$.
The next lemma is similar to \cite[Proposition 4.5]{GLPS1}.

\begin{lemma} \label{L:irraresesq}
 Any simple module $V$ in $\mathcal{D}$ with real highest weight has
  a sesquilinear structure.
  Moreover, if $v_0$ is a highest
  weight vector of $V$ there is a unique form $(\cdot,\cdot)$ on $V$
  such that $(v_0,v_0) = 1$ which is Hermitian.
\end{lemma}
\begin{proof}
  Let $V$ be a simple module in
  $\mathcal{D}$ with real weights.  Since $\dag{H_i}=H_i$, we have $\con
  V$ is a simple module with character that is the conjugate of that
  of $V^*$ (where the $\wb\U^H$-action on $\con
  V$ is given in Equation \eqref{eq:Vbar-action}).  Since the weights of the module $V$ are real, $\con V$ and
  $V^*$ have the same character so they are isomorphic.  Then Lemma
  \ref{L:sesq} applies.  Finally, the form is Hermitian by Lemma
  \ref{L:ProportionalHermitianSimple}.
 \end{proof}

Let
$\Zt=rL'_R\cap L_R$.  For $k\in
\Zt$ define $\sigma(k)$ to be the vector space
$\C$ with the action determined by $E_i, F_i,
H_i-k(H_i)$ act by zero.  This module is Hermitian with form
$(z_1,z_2)=\overline{z_1}z_2$.

\begin{proposition}\label{P:DH-Hermitian-modular}
  The category
 $\catdSeq$
   is a  sesquilinear ribbon category and a relative modular category
  with grading group
  $\Gr=\mathfrak{h}_\R^*/L_R$, symmetric small set
  $$\X=\{[\zeta]\in  \Gr  \mid \text{there exists } \alpha \in \Delta^+
  \text{ satisfying } 2\langle \alpha, \zeta \rangle\in \Z\},$$ free
  realization $\{\sigma(k)\}_{k\in rL'_R\cap L_R}$, and  bilinear map $\psi : \Gr \times \Zt \rightarrow \C^*$ given by $([\zeta], k ) \mapsto q^{2 \langle \zeta, k\rangle}$.
  \end{proposition}
\begin{proof}
By \cite[Theorem 1.3]{DGP2}, the category $\mathcal{D}$  is a relative modular category with grading group $\mathfrak{h}^*/L_R$ and the free realization given in the statement of the theorem we are proving.  Since $\mathfrak{h}_\R^*/L_R$ is a subgroup of $\mathfrak{h}^*/L_R$ this structure induces a relative modular structure  on   $\catdSeq$.
 By Lemma \ref{LAbraidherm} the dagger on $\wb\U^H$ satisfies Equation \eqref{eq:Rdag2} and so by Theorem \ref{T:SequilinearPivotal} the category $\catdSeq$  is a sequilinear ribbon $\C$-category.
\end{proof}
Let $S$ be the set of all objects of $\catdSeq$ such that the modules are simple and their forms are Hermitian. As above, such modules are called Hermitian.  Let $\catdHem$ be the full subcategory $\catdSeq$ generated by the set $S$ as defined in Section \ref{s:sesqtoHerm}.
By Lemma \ref{L:irraresesq} every simple module in $\catdSeq$ is isomorphic to a simple module endowed with a Hermitian form.  Lemma \ref{L:invol-herm}  implies that if $V,W$ are objects of $\catdHem$ and any
  $f\in\Hom_\cat(V,W)$, then $(f^{\dagger})^{\dagger} = f$.
 Finally, Proposition
\ref{P:TensorProdHerm} implies that a semi-split tensor product of two Hermitian simple objects in $\catdSeq$ is also Hermitian. Therefore, the set $S$ is a collection of Hermitian objects as defined in Section \ref{s:sesqtoHerm}.   Thus, Corollary \ref{c:hpremod} applies and we have the following theorem.

\begin{theorem}\label{T:DH-Hermitian-modular}
  The category $\catdHem$ is a
  relative Hermitian-modular category with $\Gr=\mathfrak{h}_\R^*/L_R$, $\X$, $\{\sigma(k)\}_{k\in rL'_R\cap L_R}$ and $\psi$ as in Proposition \ref{P:DH-Hermitian-modular}.
\end{theorem}
Let us increase $\X$ to
$\X'=\{[\zeta]\in \Gr \mid \text{there exists } \alpha \in \Delta^+ \text{
  satisfying } 6\langle \alpha, \zeta \rangle\in \Z\}$.  Then
$\Gr\setminus \X'$
 has no elements of order $2$ or $3$.
Recall the construction of the category $\cat^\Zt$ from Definition \ref{d:catZ}.

\begin{theorem}\label{T:DHZ-Hermitian-spherical}
  The category
  $(\catdHem)^{\Zt}$
  is a relative Hermitian-spherical
  category with grading group $\Gr=\mathfrak{h}_\R^*/L_R$, symmetric
  small set
  $$\X'=\{[\zeta]\in  \Gr  \mid \text{there exists } \alpha \in \Delta^+
  \text{ satisfying } 6\langle \alpha, \zeta \rangle\in \Z\},$$ and
  where the map $\bb$ is the constant
  $r^{-\operatorname{card}(\Delta_+)}(\operatorname{card}(L_R/\Zt))^{-1}$.
\end{theorem}
\begin{proof}
  It suffices to show that the hypothesis of Theorem
  \ref{T:relativeHer-premodularIsSpherical} are satisfied.  Theorem
  \ref{T:DH-Hermitian-modular} implies $\catdHem$ is a relative
  Hermitian-premodular category.  Let
  $\sqrt\psi:\G\times\Zt\to\Unitary$ be defined by
  $$\sqrt\psi([\gamma],\lambda)=q^{\ang{\gamma,\lambda}}.$$
  If $V$ is a module of degree $[\gamma]$, the R-matrix acts on
  $V\otimes\sigma(\lambda)$ and $\sigma(\lambda)\otimes V$ by the
  scalar $\sqrt\psi([\gamma],\lambda)$.  In particular
  $c_{\sigma(\lambda),\sigma(\mu)}=\Id_{\sigma(\lambda+\mu)}$.  Thus,
  $\sqrt\psi$ satisfies \eqref{eq:sqrpsi} and the braiding assumption
  in Equation \eqref{eq:trivialbraiding}.

  Finally, generic simple modules of $\catdHem$ are free (generated by
  their highest weight vectors) over the negative Borel of
  $U_q\mathfrak g$ whose dimension is
  $\ellr^{\operatorname{card}(\Delta_+)}$ (by the PBW theorem).  Their
  isomorphism classes in generic degree $g$ are in bijection with the
  set of weights of degree $g$ on which $L_R$ acts simply
  transitively. But two such modules are isomorphic in
  $(\catdHem)^{\Zt}$ if and only if they belong to the same $\Zt$
  orbit.  Thus there are $\operatorname{card}(L_R/\Zt)$ isomorphism
  classes in $(\catdHem)^{\Zt}$.  Since the objects of the category
  $(\catdHem)^{\Zt}$ are vector spaces (with usual direct sum and
  tensor product of vector spaces), Formula \eqref{E:defbb}
  can be used to define the map $ \bb$ which is the given constant
  map.
 \end{proof}

\begin{remark} 
It is not clear if the category of weight modules over $\wb \U$ (see \cite{GP3})  is a relative Hermitian-spherical category.  If so, it would be interesting to check if it is equivalent to the category $(\catdHem)^{\Zt}$.
\end{remark}

 \subsection{The special linear quantum Lie superalgebra }
Here we consider two categories of modules associated to the unrolled quantum group of $\mathfrak{sl}(m|n)$.  We use the notation of \cite{GPP}, see this paper for more details.

We assume that $m\neq n$ are non-zero positive integers,
$\ellr\ge \rk= m+n-1$ is an odd integer and $q$ is the root of unity
$q=\exp\bp{2i\pi/\ellr}$.  Let $A$ be the Cartan matrix of
$\mathfrak{sl}(m|n)$:
\begin{equation} \label{E:cartansuper}
A_{ii}=2 \text{ for } i \neq m, \quad
A_{mm}=0, \quad
A_{i,i\pm1}=-1 \text{ for } i \neq m, \quad
A_{m, m \pm 1}= \pm 1, \quad
A_{ij}=0 \text{ otherwise} .
\end{equation}
Define $d_i=1$ if $i\le m$ and $d_i=-1$ if $i>m$.
Then $(d_i A_{ij}) $ is a symmetric matrix.

Let $q$ an
$\ellr$th root of unity.  Let
$\roots^\pm=\roots_{\p 0}^{\pm} \cup \roots_{\p 1}^{\pm}$ to be the
set of positive and negative roots, respectively.  Here the subscripts
$\p 0$ and $\p 1$ denote the even and odd roots, respectively.  Let
$\rho_{\bar{0}}$ and $\rho_{\bar{1}}$ be half the sum of the positive
even and odd roots respectively.  Let
$\rho= \rho_{\bar{0}} -
\rho_{\bar{1}}$.  
A positive root is called simple if it cannot be decomposed into a sum
of two positive roots.  Denote by $L_R$ (the root lattice), the
$\mathbb{Z}$-module generated by the simple roots $\{\alpha_{i}\}$,
see \cite{GPP}.

\begin{definition}
Define  $U_q\mathfrak{sl}(m|n)$ as the $\mathbb{C}$-algebra with generators $E_i,F_i,K_i,K_i^{-1}$ with $i=1,...,m+n-1$ satisfying relations
\begin{align}
\label{A1}K_iK_j&=K_jK_i, &   K_iK_i^{-1}&=K_i^{-1}K_i=1,\\
\label{A2}K_iE_jK_i^{-1}&=q^{d_ia_{ij}}E_j, & K_iF_jK_i^{-1}&=q^{-d_ia_{ij}}F_j,\\
\label{A3}[E_i,F_j]&=\delta_{ij}\frac{K_i-K_i^{-1}}{q^{d_i}-q^{-d_i}}, & E_m^2&=F_m^2=0,\
\end{align}
and for $X_i=E_i,F_i$
\begin{align}
\label{A4}[X_i,X_j]=0 \quad &\mathrm{if} \; |i-j|>2,\\
\label{A5}X_i^2X_j-(q+q^{-1})X_iX_jX_i+X_jX_i^2=0 \quad &\mathrm{if} \; |i-j|=1, \; \mathrm{and} \; i\not=m,
\end{align}
\begin{align}
\label{A6}X_mX_{m-1}X_mX_{m+1}+&X_mX_{m+1}X_mX_{m-1}+X_{m-1}X_mX_{m+1}X_m\\
\nonumber &+X_{m+1}X_mX_{m-1}X_m-(q+q^{-1})X_mX_{m-1}X_{m+1}X_m=0 .
\end{align}
All generators are even except for $E_m,F_m$ which are odd and $[-,-]$ is the supercommutator given by $[x,y]=xy-(-1)^{|x||y|}yx$. $U_q(\mathfrak{sl}(m|n))$ can be given the structure of a Hopf superalgebra with coproduct $\Delta$, counit $\epsilon$ and antipode $S$ by
\begin{align}
\label{H1}\Delta (E_i)&=E_i \otimes 1+K_i^{-1} \otimes E_i, & \epsilon(E_i)&=0, & S(E_i)&=-K_iE_i,\\
\label{H2} \Delta(F_i)&=F_i \otimes K_i+1 \otimes F_i, & \epsilon(F_i)&=0, & S(F_i)&=-F_iK_i^{-1},\\
\label{H3}\Delta(K_i^{ \pm 1})&=K_i^{\pm 1} \otimes K_i^{\pm 1} , & \epsilon(K_i^{\pm 1})&=1 , & S(K_i^{\pm 1})&=K_i^{\mp 1} .
\end{align}
The unrolled quantum group $U_q^H\mathfrak{sl}(m|n)$ is the $\mathbb{C}$-algebra with generators $H_i,E_i,F_i,K_i^{\pm 1}$ with $i=1,...,m+n-1$ and relations \eqref{A1}-\eqref{A6} plus the relations
\begin{equation*}
[H_i,H_j]=[H_i,K_j^{\pm}]=0, \qquad [H_i,E_j]=a_{ij}E_j, \qquad [H_i,F_j]=-a_{ij}F_j
\end{equation*}
$$
E_\alpha^\ellr = F_\alpha^\ellr =0, \text{ for all } \alpha \in  \roots_{\p 0}^{+}
$$
where $E_{\alpha}$ and $F_{\alpha}$, for $\alpha \in \roots^{+}$, are
  the $q$-analogs
 of the Cartan-Weyl generators such that $E_{\alpha_{i}}=E_{i}$ and
 $F_{\alpha_{i}}=F_{i}$ for any simple root $\alpha_{i}$ (see Section 5.1.1 of \cite{Yam94}).
The $\mathbb{C}$-algebra  $U_q^H\mathfrak{sl}(m|n)$ becomes a Hopf algebra with coproduct, counit, and antipode satisfying relations \eqref{H1}-\eqref{H3} and
\begin{equation}\label{H4} \Delta(H_i)=H_i \otimes 1 + 1 \otimes H_i, \qquad \epsilon(H_i)=0, \qquad S(H_i)=-H_i  . \end{equation}
\end{definition}
We call a $U_q^H\mathfrak{sl}(m|n)$-module $V$ a \emph{weight module} if it is finite-dimensional and the following hold.
\begin{enumerate}
\item The elements $H_i,i=1,..,m+n-1$ act semisimply on $V$.
\item $K_i=q^{d_iH_i}$ as operators on $V$.
\end{enumerate}

We denote by
$\mathcal{D}$ the category whose objects are $U_{q}^H\mathfrak{sl}(m|n)$ weight modules with real weights where the morphisms of $\mathcal{D}$ are equivariant linear maps that respect the parity (i.e. all morphisms are even).  Let $\mathfrak{h}:=\mathrm{Span}\{H_1,...,H_{m+n-1}\}$. Any weight module $V$ admits a basis $\{v_j\}_{j \in I}$ and a family of linear functionals $\{\lambda_j\}_{j \in I} \in \mathfrak{h}^*$ such that
\[ H_iv_j=\lambda_j(H_i)v_j\]
for all $i$ and $j$. We call $\lambda_j$ the weight of $v_j$ and a weight of $V$ if it a weight for some $v \in V$. We call a weight $\lambda \in \mathfrak{h}^*$ of a module $V$ perturbative if $\lambda(H_i) \in \mathbb{Z}$ for all $i \not = m$ and a $U_{q}^H\mathfrak{sl}(m|n)$-module is called perturbative if all of its weights are. Let $\mathcal{D}^{\wp}$ be the full subcategory $\mathcal{D}$ of perturbative $U_{q}^H\mathfrak{sl}(m|n)$-modules.

Proposition 3.20 of \cite{GPP} implies both $\mathcal{D}$ and $\mathcal{D}^{\wp}$ are pivotal categories where the pivotal element is defined to be $K_{\pi}=\prod_iK_i^{n_i}\in U_{q}^H\mathfrak{sl}(m|n)$ where $\pi = 2 \rho - 2 \ellr \rho_{\bar{0}}=\sum_i{n_i\alpha_i}$ for some $n_i$.

Khoroshkin, Tolstoy \cite{KTol} and Yamane \cite{Yam94} showed that the $h$-adic completion version of quantum $\mathfrak{sl}(m|n)$ has an $R$-matrix.  This leads to the following truncated $R$-matrix and braiding in $\mathcal{D}$ and $\mathcal{D}^{\wp}$, see \cite[Section 3.2]{GPP} for details.

We follow the exposition of \cite{KTol}.
  First, let $(d_{ij})$ be the inverse of
the matrix $(a_{ij}/d_j)$.  Set
\begin{equation}
  \label{E:FormOfK}
  \HR=q^{\sum_{i,j}^{\rk}d_{ij}H_{i}\otimes H_{j}}.
\end{equation}
  Then the $R$-matrix is of the form
\begin{equation}
  \label{eq:Rh}
  {R}=\check{{R}}\HR.
\end{equation}
where $\check{{R}}$ is described as follows.   Consider the truncated quantum exponential:
  $$\exp_{q}^<(x):=\sum_{n=0}^{\ellr-1} \frac{x^n}{(n)_{q}!}.$$
  Let
\begin{equation*}
  \check{{R}}=\prod_{\alpha \in\roots^{+}}\exp_{q_{\alpha}}^<\big((-1)^{\p
    \alpha}a_{\alpha}^{-1}(q-q^{-1})(E_{\alpha}\otimes F_{\alpha})\big), \quad \quad (n)_q=\frac{1-q^n}{1-q},
\end{equation*}
where again $E_{\alpha}$ and $F_{\alpha}$ are
  the $q$-analogs
 of the Cartan-Weyl generators (see \cite[Definition 3.2]{KTol}), $q_\alpha=(-1)^{| \alpha|}q^{-\brk{\alpha,\alpha}}$ (here  $|\alpha|$ is
the parity of $E_\alpha$)  and $a_\alpha$ is defined by
 $$
 [E_\alpha,F_\alpha]=a_\alpha(q^{h_\alpha}-q^{-h_\alpha})/(q-q^{-1})
 $$
  for $\alpha \in \roots^{+}$ and $h_\alpha:=[e_\alpha,f_\alpha]$ is a relation in the Lie algebra.
The ordering in the product of $ \check{R}$ is given by a chosen fixed normal
ordering of $\Delta^{+}$.

Recall our conventions for antiautormorphisms from \eqref{eq:dag2}.
\begin{definition}
 Consider the bijection $\dagger$ of $U_{q}^{H}(\mf{g})$ defined by
\begin{align}
  \dagger(E_i) = F_i \quad \text{$i\neq m$}, \quad \dagger(E_m) = -F_m,
  \quad \dagger(F_i) = E_i, \quad \dagger(K_i) = K_i^{-1}, \quad \dagger(H_i) = H_i.
\end{align}
Observe that this map is a superinvolution.
\end{definition}

\begin{lemma}\label{L:antiautomorphismUHgSuper}
  The map
  $\dagger \maps U_{q}^{H}(\mf{g}) \to U_{q}^{H}(\mf{g})$ is a
  anti-superautomorphism as in Equation \eqref{eq:dag2}.
\end{lemma}

\begin{proof}
  First observe that the effect of the flip map and superflip map $\tau$
  on the output of $\Delta$ on any generator is identical.  It is
  therefore straightforward as in the non-super case to see that
  $\Delta(X^{\dagger}) = (\tau \Delta(X))^{\dagger}$ for all
  generators $X$. \end{proof}

\begin{lemma} \label{SLAbraidherm}
One has the identity:
\begin{equation}
  (\dagger\otimes\dagger)\bp{R}=\tau(R^{-1}).
 \end{equation}
 \end{lemma}

 \begin{proof}
 The result follows easily from
$ \left( \exp_{q_{\alpha}}(X) \right)^{-1}=
\exp_{q^{-1}_{\alpha}}(-X) $
and the fact that $\dagger(E_\alpha)=\pm F_{\alpha}$ and $\dagger(F_\alpha)=\pm E_{\alpha}$.  Note that in the super case, a sign arising from $\tau$ gets canceled by a sign coming from the definition of $\dagger$ on odd elements.
 \end{proof}

\begin{lemma} \label{SLApivotlemma}
The dagger of the pivotal element is given by $K_{\pi}^{\dagger} = K_{\pi}^{-1} $.
\end{lemma}

\begin{proof}
This follows trivially from the definition of $\dagger$ on this super Hopf algebra.
\end{proof}

\begin{proposition}
The category $\mathcal{D}$ of weight modules is a $\mathbb{C}$-sesquilinear ribbon category.
\end{proposition}

\begin{proof}
By \cite{GPP}, $\mathcal{D}$ is a ribbon category.  Lemmas \ref{SLAbraidherm} and \ref{SLApivotlemma} show that the dagger involution is compatible with the braiding and pivotal structure, and consequently the ribbon structure.
\end{proof}
\begin{lemma} \label{L:sirraresesq}
 Any simple module $V$ in $\mathcal{D}$ with real highest weight has
  a sesquilinear structure.   Moreover, if $v_0$ is a highest
  weight vector of $V$ there is a unique form $(\cdot,\cdot)$ on $V$
  such that $(v_0,v_0) = 1$ which is Hermitian.
\end{lemma}
\begin{proof}
  The isomorphism class of a simple module $V$ of
  $\mathcal{D}$ is uniquely determined by its highest weight and the
  parity of its highest weight vector. Since the weights of the module
  $V$ are real, $\con V$ and
  $V^*$ have the same character.  Also, the highest weight space of
  $\con V$ and
  $V^*$ both have the parity of the lowest weight space of
  $V$, so $\con V\simeq
  V^*$ with an even isomorphism. Then Lemma \ref{L:sesq} applies.
  Finally, the form is Hermitian by Lemma
  \ref{L:ProportionalHermitianSimple}.
\end{proof}

Following \cite{GPP}, let
$L_W$ be the \emph{weight lattice}, which is the
$\Z$-lattice generated by the fundamental dominant weights.  Let
$\Gr$ be the group $\mathfrak{h}^*_\R/L_R\cong
(\Unitary)^{m+n-1}$ where
$\mathfrak{h}^*_\R=L_W\otimes_\Z\R$ is the set of real weights.  Let
$\Xs=\{g\in\Gr:\cat_g\text{ is not semisimple or
}6g=0\}$.  Then
$\Gr\setminus\Xs$ is a Zariski open dense subset of
$\Gr$ and in particular $\Xs$ is small. Let
$\Zt=\Z/2\Z\times\Zt_0$ where $\Zt_0=\ellr L_W\cap
L_R$.  For $z=(p,\lambda)\in\Zt$,
$\sigma(z)$ is the unique module structure on the vector space
$\C$ with parity $p$, weight $\lambda$ and Hermitian form $(v\mid
v')=\bar vv'$.  The bilinear pairing
$\psi:\Gr\times\Zt\to\C$ is given by $\psi([\zeta],(p,k))=q^{2 \langle
  \zeta, k\rangle}$.
 The m-trace on projective modules is discussed in \cite[Section 3.7]{GPP}.
 The relative modularity of
$\catdSeq$ follows from \cite[Theorem 3.37]{GPP}.  Equation
\eqref{eq:trivialbraiding} applies if we restrict the translation
group to
$\Zt_0\simeq\{0\}\times\Zt_0\subset\Zt$.  The rest is similar to the
Lie algebra case and we have the following.
\begin{theorem} \label{T:superex}
  With the above data associated to $U_q\mathfrak{sl}(m|n)$ at an odd root of unity we have the following.
  \begin{enumerate}
  \item The category $\catdSeq$ of sesquilinear real weight module is ribbon sesquilinear.
  \item Its subcategory $\catdHem$ generated by simple Hermitian
    modules a is relative Hermitian-modular giving rise to a Hermitian
    $(2+1)$-TQFT.
  \item The associated category $(\catdHem)^{\Zt_0}$ with the map $\bb$ is given by \eqref{E:defbb}, is relative
    Hermitian-spherical giving rise to a relative $\Gr$-LW-system.
  \end{enumerate}
\end{theorem}

Next we define a relative premodular category from the subcategory of
perturbative modules by taking a quotient with respect to the
negligible morphism of a different m-trace than the m-trace used
above, and on a different
ideal than $\Proj$, see \cite[Section 4.3]{GPP} for more details.
An interesting feature of the associated
ribbon graph invariants is that they are   specializations at root of unity of invariants coming from quantum
$U_q\mathfrak{sl}(m|n)$ for generic $q$ and in particular they   are related to the perturbative expansion after setting $q=\exp(h)$.
Let us quickly recall
the categories involved.

There exists a replete collection of
objects in $\mathcal{D}^{\wp}$ called negligible, see Definition 5.5 of \cite{GPP}. A morphism $f$ of
$\mathcal{D}^{\wp}$ is called negligible if it factors through a
negligible object.  Note that we automatically have that $f^\dagger$
is negligible if and only if $f$ is negligible.  Then
$\mathcal{D}^{\aleph}$ is a quotient of a full subcategory of
$\mathcal{D}^{\wp}$ by the negligible morphisms
($\Hom_{\mathcal{D}^{\aleph}}(V,W)$ is the quotient of
$\Hom_{\mathcal{D}^{\wp}}(V,W)$ by the vector space of negligible
morphisms).  Let $\Zt^\wp=\Z/2\Z\times\Zt^\wp_0\subset\Zt$ where
$\Zt^\wp_0\simeq\Z$ is the subgroup of $\Zt_0$ generated by
$\frac{\left|m-n\right|}{\gcd(m,n)}\ellr$ times the
$m$th fundmental
 weight.   Then the realization $\sigma$ induces
by restriction a functor
$\sigma^{\aleph}:\Zt^\wp\to \mathcal{D}^{\aleph}$. The categories
$\mathcal{D}^{\wp}$ and $\mathcal{D}^{\aleph}$ are $\Unitary$-graded
where a module is of degree $z$ if $K_m^\ellr$ acts by the scalar $z$
on it. Let $\Xs=\{\pm1\}$.
\begin{theorem}\cite[Theorem 5.17]{GPP}
  $\mathcal{D}^{\aleph}$ is a premodular $\Unitary$-category relative to $(\Zt^\wp,\Xs^\wp)$.
\end{theorem}

Then Equation \eqref{eq:trivialbraiding} holds if we
restrict the translation group to
$\Zt_0\simeq\{0\}\times\Zt_0\subset\Zt$.  Moreover, $\sqrt\psi:\Gr\times\Zt\to\Unitary$ defined by $\sqrt\psi([\zeta],(p,k))=q^{ \langle
  \zeta, k\rangle}$ satisfies Equation \eqref{eq:sqrpsi}.  Thus, the category $((\mathcal{D}^{\aleph})^\Herm)^{\Zt_0}$ satisfies all the hypothesis of Theorem \ref{T:relativeHer-premodularIsSpherical} except the existence of a map
  $\bb$.  We are lead to the following conjecture.
\begin{conjecture}\label{c:bb=qdim}
  The category $((\mathcal{D}^{\aleph})^\Herm)^{\Zt_0}$ admits a map
  $\bb$ satisfying Equation \eqref{eq:bb}. A candidate is given on
  generic simple $V$ of degree $g$ by
  $$\bb(V)=\operatorname{oqdim}(V)/\sum_{V'}\operatorname{oqdim}(V')^2,$$
  where the sum is over isomorphism classes of simples of degree $g$
  and oqdim is the ordinary (non super) quantum dimension of the
  underlying $U_q\mathfrak{g}_{\bar0}$-module given by
  $\operatorname{oqdim}(V)=\tr_V(K_{2\rho_{\bar0}})$.
\end{conjecture}
In \cite{GPA}, Conjecture \ref{c:bb=qdim} is shown to hold for
$\mathfrak{sl}(2|1)$.
\begin{theorem}\label{T:sl21ConjTrue}
  If Conjecture \ref{c:bb=qdim} is true, the category
  $((\mathcal{D}^{\aleph})^\Herm)^{\Zt_0}$ is relative
  Hermitian-spherical giving rise to a relative
  $\Unitary$-LW-system.

  In particular,   $((\mathcal{D}^{\aleph})^\Herm)^{\Zt_0}$ is relative
  Hermitian-spherical giving rise to a relative
  $\Unitary$-LW-system for $\mathfrak{sl}(2|1)$.
\end{theorem}

\end{document}